\documentclass[12pt,reqno]{amsart}

\usepackage[T1]{fontenc}
\usepackage{amsthm,mathtools,amssymb}
\usepackage{mathpazo}
\usepackage{enumerate}
\usepackage[colorlinks=true,linkcolor=blue,citecolor=green,urlcolor=blue]{hyperref}

\usepackage{tikz}

\advance\textwidth by 1.2in
\advance\oddsidemargin by -.6in
\advance\evensidemargin by -.6in

\newtheorem{theorem}{Theorem}[section]
\newtheorem{proposition}[theorem]{Proposition}
\newtheorem{lemma}[theorem]{Lemma}
\newtheorem{corollary}[theorem]{Corollary}
\theoremstyle{definition}
\newtheorem{definition}[theorem]{Definition}
\newtheorem{example}[theorem]{Example}
\theoremstyle{remark}
\newtheorem{remark}[theorem]{Remark}

\newcommand{\mf}{\mathfrak}
\newcommand{\mcal}{\mathcal}
\newcommand{\G}{\mathcal G}
\newcommand{\al}{\alpha}
\newcommand{\be}{\beta}
\newcommand{\ga}{\gamma}
\newcommand{\la}{\lambda}
\newcommand{\La}{\Lambda}
\newcommand{\Z}{\mathbb Z}
\newcommand{\Q}{\mathbb Q}
\newcommand{\C}{\mathbb C}
\newcommand{\eps}{\varepsilon}
\newcommand{\Mob}{\mu_{\mathrm{Mob}}}
\newcommand{\bzero}{\mathbf 0}
\newcommand{\mvec}{\mathbf m}
\newcommand{\bk}{\mathbf k}
\newcommand{\bx}{\mathbf x}
\newcommand{\by}{\mathbf y}

\DeclareMathOperator{\mult}{mult}
\DeclareMathOperator{\supp}{supp}
\DeclareMathOperator{\htop}{ht}
\DeclareMathOperator{\Sym}{Sym}

\newcommand{\im}{\rm{im}}
\newcommand{\re}{\rm{re}}
\newcommand{\dis}{\displaystyle}

\numberwithin{equation}{section}
\makeatletter
\def\section{\def\@secnumfont{\mdseries}\@startsection{section}{1}%
  \z@{.7\linespacing\@plus\linespacing}{.5\linespacing}%
  {\normalfont\scshape\centering}}
\def\subsection{\def\@secnumfont{\bfseries}\@startsection{subsection}{2}%
  {\parindent}{.5\linespacing\@plus.7\linespacing}{-.5em}%
  {\normalfont\bfseries}}
\makeatother
\title[Attached imaginary graphs]
{Chromatic characters and root multiplicities for Borcherds--Kac--Moody algebras}
\author{Priyanshu Chakraborty}
\address{Indian Institute of Science, Bangalore, India}
\email{priyanshuc@iisc.ac.in}
\thanks{PC thanks the IISc IPDF  postdoctoral research fellowship for carrying out this research.}
\author{Supriya Saha}
\address{Indian Institute of Science, Bangalore, India}
\email{supriyasaha@iisc.ac.in}
\thanks{SS thanks the NBHM PhD fellowship  for carrying out this research.}
\author{R. Venkatesh}
\address{Indian Institute of Science, Bangalore, India}
\email{rvenkat@iisc.ac.in}
\thanks{RV’s work is partially supported by CRG/2023/008780 of ANRF and DST FIST program-2021 [TPN700661].}
\date{}
\hypersetup{pdftitle={Chromatic characters and root multiplicities for Borcherds--Kac--Moody algebras},
  pdfauthor={Priyanshu Chakraborty, Supriya Saha, R. Venkatesh}}

\begin{document}

\begin{abstract}
The main aim of this article is to obtain explicit formulas for the root multiplicities of a class of finite-rank symmetrizable Borcherds--Kac--Moody algebras $\mf g(A)$ whose real part is of type $A$. The Dynkin diagram of each such algebra is obtained from the disjoint union of a Dynkin diagram of type $A$ and a finite simple graph $\mcal G_{\mathrm{im}}$, whose vertices correspond to imaginary simple roots, by choosing one vertex in each and joining the chosen vertices by a single edge of multiplicity $c\geq 1$.
More generally, we consider the same construction with the type $A$ diagram replaced by the Dynkin diagram of an arbitrary finite-rank symmetrizable Kac--Moody algebra $\mf g_{\mathrm{re}}$. The resulting symmetrizable Borcherds--Kac--Moody algebra $\mf g(A)$ has real part $\mf g_{\mathrm{re}}$ and imaginary part represented by $\mcal G_{\mathrm{im}}$.

Our starting point is a graph-theoretic description of the quotient of the denominator products of $\mf g(A)$ and $\mf g_{\mathrm{re}}$. This quotient is a multivariate independence polynomial of $\mcal G_{\mathrm{im}}$, with the variable corresponding to the attachment vertex weighted by normalized highest-weight characters of $\mf g_{\mathrm{re}}$. 
We prove that the coefficients of powers of this quotient factor into a generalized chromatic polynomial and a tensor-power weight multiplicity.
Using this, we obtain explicit formulas for all root multiplicities with nonzero imaginary component in terms of generalized chromatic polynomials of $\mcal G_{\mathrm{im}}$ and weight multiplicities in tensor products of highest-weight modules of $\mf g_{\mathrm{re}}$.  When $\mcal G_{\mathrm{im}}$ is chordal, these formulas specialize to finite divisor sums involving products of binomial coefficients; paths and complete graphs provide particular examples. In particular, we obtain explicit formulas for the root multiplicities of $\mf g(A)$ when its real part $\mf g_{\mathrm{re}}$ is of type $A$.


\end{abstract}
\maketitle
\pagestyle{plain}

\section{Introduction}\label{sec:introduction}

The denominator identity connects the root multiplicities of a
Kac--Moody algebra with the action of its Weyl group. For
Borcherds--Kac--Moody algebras, the corresponding identity also
contains correction terms indexed by sets of mutually orthogonal
imaginary simple roots. These terms provide a natural connection
with independent sets of graphs. In this paper, we use this
connection to compute root multiplicities for a family of algebras
obtained by attaching an arbitrary imaginary graph to a
symmetrizable Kac--Moody Dynkin diagram, by choosing one vertex in each and joining the chosen vertices by a single edge of multiplicity $c\geq 1$.

 Let $\mcal G_{\mathrm{re}}$ be the Dynkin diagram of a
finite-rank symmetrizable Kac--Moody algebra $\mf g_{\mathrm{re}}$, with simple
roots $\al_{i}$, $i\in {I^{\mathrm{re}}}$. Let $\G_{\mathrm{im}}$ be a finite simple graph on
$I^{\mathrm{im}}=\{1,2,\dots,n\}$, and associate an isotropic imaginary simple
root $\be_i$ to each vertex $i\in I^{\mathrm{im}}$. Choose $u\in I^{\mathrm{re}}$ and $v\in I^{\mathrm{im}}$,
and join $u$ to $v$ by a single edge with multiplicity $c\ge 1$. The real-row Cartan entry at
this edge is $-c$; thus
\[
a_{i,\be_{j}}=-c\,\delta_{i,u}\delta_{j,v}.
\]
The remaining entries are chosen to define a symmetrizable
Borcherds--Cartan matrix $A$, with imaginary adjacency graph $\G_{\mathrm{im}}$.
All simple roots have multiplicity one. We work with ordinary Lie
algebras and with a realization in which both the simple roots and the simple
coroots are linearly independent. 

Let $D_{\mf g}$ and $D_{\mf g_{\mathrm{re}}}$ be the normalized denominator products of
$\mf g(A)$ and $\mf g_{\mathrm{re}}$, respectively, and set
\[
S=\frac{D_{\mf g}}{D_{\mf g_{\mathrm{re}}}}
=\prod_{\xi\in\Psi}(1-e^{-\xi})^{\mult(\xi)},
\qquad
\Psi=\Delta^+(\mf g(A))\setminus\Delta^+_{\mathrm{re}}.
\]
where $\Delta^+(\mf g(A))$ is the set of positive roots of $\mf g(A)$ and $\Delta^+_{\mathrm{re}}$ is the set of positive roots $\mf g_{\mathrm{re}}$.
Choose a dominant integral weight $\la$ satisfying
$\la(h_{\alpha_i})=c\delta_{i,u}$, $i\in I^{\mathrm{re}}$, and write
\[
H_\la=e^{-\la}\operatorname{Ch} V_{\mf g_{\mathrm{re}}}(\la),\qquad
M_\la(\eta;t)=H_\la^t[e^{-\eta}]
=\dim\left(V_{\mf g_{\mathrm{re}}}(\la)^{\otimes t}\right)_{t\la-\eta}.
\]
The normalized character is independent of the choice of the
remaining Cartan values of $\la$.
Using the substitution $x_{j}=e^{-\be_{j}}$, the following identity
\[
S=I(\G,\mathbf z),\qquad
z_v=-x_vH_\la,\qquad z_{j}=-x_{j}\quad(j\neq v),
\]
is proved in Proposition~\ref{cor:S}. Here $I(\G,\bx)$ denotes the
multivariate independence polynomial. For $\mvec=(m_{j}:j\in I^{\mathrm{im}})\in\Z_{\geq0}^{I^{\mathrm{im}}}$, let
$\pi_{\G_{\mathrm{im}}}^{\mvec}(q)$ be the generalized chromatic polynomial:
it counts assignments of $m_{j}$ distinct colors to vertex $j$ from
a palette of $q$ colors, with disjoint color sets at adjacent
vertices. The identity
\[
I(\G,\bx)^q=\sum_{\mvec\geq0}\pi_{\G}^{\mvec}(q)\bx^{\mvec}
\]
holds coefficientwise with $q$ indeterminate. Combining it with
the expression for $S$ gives our main coefficient formula
(Corollary~\ref{cor:coefficient}): for
$\eta=\sum_{i\in {I^{\mathrm{re}}}}k_i\al_i\in Q_{\mathrm{re}}^+$ and
$\mu=\eta+\sum_{j\in I^{\im}}m_{j}\be_{j}$,
\[
{
S^q[e^{-\mu}]
=(-1)^{|\mvec|}M_\la(\eta;m_v)\pi_{\G_{\mathrm{im}}}^{\mvec}(q).
}
\]
This formula gives an explicit
graph-polynomial expression of Kang's character-based
denominator formula, see \cite{Kang}.

Root multiplicities follow by comparing linear coefficients in
$q$. For $\mvec\neq\bzero$, put
\[
\kappa_{\G_{\mathrm{im}}}(\mvec)=(-1)^{|\mvec|-1}\pi_{\G_{\mathrm{im}}}^{\mvec}(q)[q],
\qquad
g=\gcd\bigl(\{k_i:i\in {I^{\mathrm{re}}}\}\cup\{m_{j}:j\in I^{\mathrm{im}}\}\bigr).
\]
Then considering $-S^q[e^{-\mu}][q]$ and using M\"obius inversion we get
(Theorem~\ref{cor:root-general})
\[
{
\mult(\mu)
=\sum_{d\mid g}\frac{\Mob(d)}d
M_\la\left(\frac\eta d;\frac{m_v}d\right)
\kappa_{\G_{\mathrm{im}}}\left(\frac\mvec d\right).
}
\]
Here $\Mob$ is the arithmetic M\"obius function, and root
multiplicity is extended by zero to nonroots.
When $\G_{\mathrm{im}}(\supp(\mvec))$ is disconnected, its generalized chromatic
polynomial has no linear term, and the corresponding root
multiplicity vanishes.\\

For chordal graph $\G$, $\pi_{\G}^{\mvec}(q)$ admit particularly simple
expressions using that we can get an closed formula for the $\mult(\mu).$ In particular, for a perfect elimination ordering in which earlier
neighbors form a clique, set
\[
a_r(\mvec)=\sum_{\substack{s\leq r\\s\in N_{\G}[r]}}m_s,
\]
where $N_{\G}[r]$ is the closed neighborhood of $r$. The
multicoloring product of \cite{BHV} is
\[
\pi_{\G}^{\mvec}(q)
=\prod_{r=1}^n\binom{q-a_r(\mvec)+m_r}{m_r}.
\]
If $J=\supp(\mvec)$ is nonempty and connected and $r_0=\min J$,
we obtain
\[
\kappa_{\G_{\mathrm{im}}}(\mvec)
=\frac1{m_{r_0}}
\prod_{r\in J\setminus\{r_0\}}\binom{a_r(\mvec)-1}{m_r}.
\]
Consequently, every coefficient of $S^q$ is a product of binomial
coefficients times one tensor-power weight multiplicity. Every root
multiplicity with nonzero imaginary simple-root component is then
given by a finite divisor sum of such products. For the path $L_n$, one has $a_r(\mvec)=m_{r-1}+m_r$, with $m_0=0$,
and hence these expressions give $\mult(\mu)$ for the case when $\mcal G_{\mathrm{im}}=L_n$.

We also give an explicit formula for $M_\la(\eta,;t)$. For a general
finite-rank symmetrizable  algebra $\mf g_{\mathrm{re}}$, using some Weyl group arguments we proved an combinatorial formula for $M_\la(\eta;t)$ in Proposition \ref{prop:finite-Weyl}. For
$\mcal G_{re}=A_\ell$ and $c=1$, the exterior-power realization of
$V_{\mf g_{\mathrm{re}}}(\omega_u) \cong \bigwedge^u(\C^{l+1})$ gives a finite multinomial expression for $H_{\omega_u
}(\mathbf y)$, Proposition \ref{prop:A-character}. At the endpoint $u=\ell$, it gives
\[
M_{\omega_\ell}(\eta;t)
=\frac{t!}{k_1!\,\prod_{i=2}^\ell(k_i-k_{i-1})!\,(t-k_\ell)!}
\]
when $0\leq k_1\leq\cdots\leq k_\ell\leq t$, and is zero
otherwise. We treat arbitrary $c$ at this endpoint through the
symmetric-power realization of $V_{\mf g_{\mathrm{re}}}(c\omega_\ell)$. These
computations apply to every attached graph $\G_{\mathrm{im}}$ and give explicit
factorial formulas in the chordal and path cases.

The paper is organized as follows:
Section~\ref{sec:setup} fixes the Cartan data and normalized
character, and Section~\ref{sec:chromatic} recalls the multicoloring
identities. Sections~\ref{sec:quotient} and~\ref{sec:multiplicities}
specialize the denominator and multiplicity identities to a single
attachment. Sections~\ref{sec:chordal} treat
chordal graphs and paths. Section~\ref{sec:real-factor} gives finite
expressions for $M_\la(\eta,t)$. Section~\ref{sec:type-A}
develops the type $A$ formulas and numerical examples.

\section{Attached graph and normalized character}\label{sec:setup}

\subsection{Cartan Data and Borcherds--Kac--Moody algebra}
All vector spaces are assumed to be defined over complex numbers. A Borcherds--Cartan matrix  $A=(a_{ij})_{i,j\in I}$
 indexed by a finite set
$I$ is a real matrix  satisfying the following conditions for all $i,j\in I$:
\begin{enumerate}
    \item  there exists a diagonal matrix
$D=\operatorname{diag}(d_1,\ldots,d_{|I|})$ with positive entries such that $DA$
is symmetric (i.e., $A$ is symmetrizable);
    \item $a_{ii}=2$ or $a_{ii}\leq 0$;
    \item $a_{ij}\leq 0$ if $i\neq j$, and $a_{ij}\in\mathbb{Z}$
          if $a_{ii}=2$;
    \item $a_{ij}=0$ if and only if $a_{ji}=0$.
\end{enumerate}
Denote by $I^{\mathrm{re}}=\{i\in I : a_{ii}=2\}$ and
$I^{\mathrm{im}}=I\setminus I^{\mathrm{re}}$.
The \emph{Borcherds--Kac--Moody algebra} $\mathfrak{g}=\mathfrak{g}(A)$
associated to the Borcherds--Cartan matrix $A$ is the Lie algebra
generated by $e_i,f_i,h_{\alpha_i}$, $i\in I$, with the following defining relations:
\begin{align*}
\text{(R1)}\quad
& [h_{\alpha_i},h_{I^{\mathrm{re}}}]=0
&& \text{for } i,j\in I;\\
\text{(R2)}\quad
& [h_{\alpha_i},e_k]=a_{ik}e_k,\qquad [h_{\alpha_i},f_k]=-a_{ik}f_k
&& \text{for } i,k\in I;\\
\text{(R3)}\quad
& [e_i,f_{j}]=\delta_{ij}h_{\alpha_i}
&& \text{for } i,j\in I;\\
\text{(R4)}\quad
& (\operatorname{ad}e_i)^{1-a_{ij}}e_{j}=0,\qquad
  (\operatorname{ad}f_i)^{1-a_{ij}}f_{j}=0
&& \text{if } i\in I^{\mathrm{re}} \text{ and } i\ne j;\\
\text{(R5)}\quad
& [e_i,e_{j}]=0,\qquad [f_i,f_{j}]=0
&& \text{if } a_{ij}=0.
\end{align*}

Denote the set of simple roots, simple coroots corresponding to real nods by $\alpha_i$ and $h_{\alpha_i}$ $(i\in I^{\mathrm{re}})$. We also denote 
set of simple roots, simple coroots corresponding to real nods by $\alpha_i$ and $h_{\alpha_i}$ $(i\in I^{\mathrm{re}})$. We use $\beta_i$ and $h_{\beta_i}$ $(i\in I^{\mathrm{im}})$ to denote the
set of simple roots, simple coroots corresponding to imaginary nodes. We also use $h_i$ for $i\in I$ to denote the simple coroots if we do not need to specify the real or imaginary nodes specifically. 
The Weyl group of $\mf g$ is denoted by $W_I$. The real part of $\mf g$, denoted as $\mf g_\mathrm{re}$, is the subalgebra of $\mf g$ generated by $e_i, f_i\in I^\mathrm{re}$.
Set
\[
Q^+=\bigoplus_{i\in {I^{\mathrm{re}}}}\Z_{\geq0}\al_i\oplus \bigoplus_{i\in {I^{\mathrm{im}}}}\Z_{\geq0}\beta_i,
\qquad \Delta^+ =\Delta^+(\mf g), \qquad \Delta_\mathrm{re}^+=\Delta^+(\mf g_\mathrm{re})
\]
where $\Delta^+(\mf a)$ denotes the set of positive roots of $\mf a$. We also use $Q^+_{\mathrm{re}}=\bigoplus_{i\in {I^{\mathrm{re}}}}\Z_{\geq0}\al_i$.
The quasi-Dynkin diagram $\mathcal{G}$ of $\mf g$ is a simple graph with vertex $I$ and we draw an edge between two distinct vertices $i, j\in I$ if $a_{ij}\neq 0.$ The edge set of $\mathcal{G}$ is denoted by $E.$
For any subset $S\subseteq I$, denote $A_S = (a_{ij})_{i, j\in S}$. We denote $\mathcal{G}_{\mathrm{re}}$ (respectively $\mathcal{G}_{\mathrm{im}}$) by the subgraph spanned by $I^{\mathrm{re}}$ (respectively $I^{\mathrm{im}}$).
\subsection{The special case} \textit{ Throughout this paper, we consider only the following special case.} The quasi-Dynkin diagram $\mathcal{G}$ of $\mf g$ is obtained from a quasi-Dynkin diagram $\mathcal{G}_{\mathrm{re}}$ of a symmetrizable Kac-Moody algebra by attaching arbitrary finite simple graph $\mathcal{G}_{\mathrm{im}}$ by choosing one vertex from $\mathcal{G}_{\mathrm{re}}$ and one vertex from  $\mathcal{G}_{\mathrm{im}}$ and joining them by an edge (with multiplicity $c\ge 1$). We level the vertex set of $\G_{\im}$ by $\{1,2,\dots,n\}.$

\begin{center}

\begin{tikzpicture}[thick]
  \draw[blue!60!black, fill=blue!8]
    plot[smooth cycle, tension=0.8] coordinates {
      (-0.85,0.10) (-0.65,0.75) (-0.10,0.65)
      (0.50,0.85) (0.85,0.30) (0.65,-0.25)
      (0.35,-0.70) (-0.30,-0.60) (-0.75,-0.45)
    };

  \draw[teal!60!black, fill=teal!8]
    plot[smooth cycle, tension=0.8] coordinates {
      (2.15,0.20) (2.45,0.80) (3.00,0.65)
      (3.55,0.75) (3.85,0.20) (3.60,-0.25)
      (3.35,-0.70) (2.80,-0.55) (2.30,-0.65)
    };

\draw (0,0) -- node[midway, above] {$c$} (3,0);
\fill (0,0) circle (2pt);
  \fill (3,0) circle (2pt);

  \node[below=3pt] at (0,0) {$u$};
  \node[below=3pt] at (3,0) {$v$};

  \node[anchor=north] at (0,-0.85) {$\mathcal{G}_{\mathrm{re}}$};
  \node[anchor=north] at (3,-0.85) {$\mathcal{G}_{\mathrm{im}}$};
\end{tikzpicture}

\end{center}
This special can be described as follows: We fix two vertices $u, v\in I$ such that $u\in I^{\mathrm{re}}$  and $v\in  I^{\mathrm{im}}$, and they are adjacent in $\mathcal{G}$. We assume that the Borcherds-Cartan matrix of $\mathfrak g$ is of the form \begin{align}\label{eqn: matix general }
A=
\begin{pmatrix}
	A_{I^{\mathrm{re}}} & B\\
	B^{T} & C
\end{pmatrix},
\end{align} 
where $B=(b_{ij})_{i \in I^{\re}, j \in I^{\im}}$ and $b_{ij}=-c \delta_{i,u}\delta_{j,v}$ and 
$C=(c_{ij})_{i,j \in I^{\im}},$ where $c_{ij}=-1$ only when there is an edge between $i,j$ in $\mathcal{G}$ and zero otherwise.
It is enough to have $A_{I^{\mathrm{re}}}$ symmetrizable, then there exists  $D'=\operatorname{diag}(d_i)_{i\in I^{\mathrm{re}}}$ which makes $D'A_{I^{\mathrm{re}}}$ a symmetric matrix. Then with the choice of the diagonal matrix  \[
D=
\operatorname{diag}
\left(
(d_i)_{i\in I^{\mathrm{re}}},
\underbrace{d_u,\ldots,d_u}_{|I\backslash I^{\mathrm{re}}| \text{ times} }
\right), 
\]
$DA$ is a symmetric matrix.
We fix a realization in which both the simple roots and the simple
coroots are linearly independent. Fix a weight
$\rho$ satisfying:
\[
\rho(h_{\alpha_i})=1\quad(i\in I^{\re}),\qquad
\rho(h_{\be_j})=0\quad(j\in I^{\im}).
\]
In the realization of $\mf g_{\re}$, choose $\rho'$ with
$\rho'(h_{\alpha_i})=1$ for every $i\in {I^{\mathrm{re}}}$. Put
\[
Q^+=Q_{\mathrm{re}}^+\oplus\bigoplus_{j\in {I^{\mathrm{im}}}}\Z_{\geq0}\be_j,
\qquad
\Psi=\Delta^+(\mf g(A))\setminus\Delta^+({I^{\mathrm{re}}}).
\]
For a nonzero $\xi\in Q^+$, let $\mult(\xi)=\dim\mf g(A)_\xi$.
\subsection{Character of highest weight $\mf g_{\mathrm{re}}$ modules}\label{subsec: ch of g_X}
Let $\Lambda_u$ denote the fundamental highest weight corresponding to the node $u$. Then we have 
\[
\Lambda_u(h_{\alpha_i})=\delta_{i,u}\qquad(i\in {I^{\mathrm{re}}}).
\]
Set $\la=c\La_u$, a dominant integral weight of $\mf g_{\mathrm{re}}$. Let $V_{\mf g_{\mathrm{re}}}(\la)$ be the irreducible
integrable highest-weight module for $\mf g_{\mathrm{re}}$ with highest weight $\la$ and define
\begin{equation}\label{eq:H}
H_\la=e^{-\la}\operatorname{Ch} V_{\mf g_{\mathrm{re}}}(\la)
=\sum_{\ga\in Q_{\mathrm{re}}^+}h_\la(\ga)e^{-\ga},
\qquad h_\la(\ga)=\dim V_{\mf g_{\mathrm{re}}}(\la)_{\la-\ga}.
\end{equation}
Clearly, $h_\la(0)=1$.
It is well known that, in finite type $H_\la$ is a polynomial; otherwise it is a formal series.\\

We use the standard independence of the normalized highest-weight
character from the choice of the remaining Cartan values; see
\cite[Chapters~9--10]{Kac}. Explicitly, if $\la'(h_{\alpha_i})=\la(h_{\alpha_i})$, for all $i \in {I^{\mathrm{re}}}$,
then $\zeta=\la'-\la$ vanishes on the simple coroots of $\mathfrak g_{\mathrm{re}}$. Hence
\[
V_{\mf g_{\mathrm{re}}}(\la')\cong V_{\mf g_{\mathrm{re}}}(\la)\otimes\C_\zeta, \text{ and }
\qquad  H_{\la'}=e^{-\la'}\operatorname{Ch} V_{\mf g_{\mathrm{re}}}(\la')=H_\la,
\]
where $\C_\zeta$ is the one-dimensional module with Cartan weight
$\zeta$ and trivial action of all simple root generators.

\subsection{Completions and coefficients}
Write $y_i=e^{-\alpha_i}$ for $i\in I^{\mathrm{re}}$ and $x_{j}=e^{-\be_j}$ for $j\in I^{\mathrm{im}}$. For
$\mvec=(m_i:i\in I^{\im})\in\Z_{\geq0}^{I^{\im}}$, put
\[
|\mvec|=\sum_{i\in I^{\im}}m_i, \qquad
\mvec!=\prod_{i\in I^{\im}}m_i!.
\]
If $\eta(\mathbf{k})=\sum_{i\in {I^{\mathrm{re}}}}k_i\al_i$ and $\nu(\mvec)=\sum_{j\in I^{\mathrm{im}}}m_j\be_j$, then we have 
 $e^{-(\eta(\mathbf{k})+\nu(\mvec))}=\by^{\bk} \bx^{\mvec}$. For notational convenience some times we write $\eta(\mathbf k)$ just by $\eta.$
All series are interpreted in
\[
\widehat R=\Q[[y_i, x_j:i\in {I^{\mathrm{re}}},\ j\in {I^{\mathrm{im}}}]].
\]
For a series $F$, denote
$F[\bx^{\mvec}]$, $F[e^{-\mu}]$, and $F[e^{-\mu}][q]$ by the coefficients of $\bx^{\mvec}$ in $F$ and so on. 
If $U \in \widehat R$ has zero constant term, then define
\begin{equation}\label{eq:formal-power}
(1+U)^q=\sum_{k\geq0}\binom qk U^k,
\qquad
\binom qk=\frac{q(q-1)\cdots(q-k+1)}{k!},\qquad \binom q0=1.
\end{equation}
For each fixed monomial, only finitely many terms contribute, so its
coefficient is a polynomial in the indeterminate $q$.
For indeterminate $q$, these series belong to
$\Q[q][[y_i, x_j:i\in {I^{\mathrm{re}}},\ j\in {I^{\mathrm{im}}}]]$.

\section{Generalized chromatic polynomials}\label{sec:chromatic}

\subsection{Graph notation and multicoloring} For this section, assume that $\mathcal{G}$ is an arbitrary finite simple graph with vertex set $I$.
We use $\G_S$ for the subgraph induced by $S\subseteq I$, and
$\mcal I(\G)$ for the set of independent subsets of $\G$, including
$\varnothing$. Put
\[
x_S=\prod_{i\in S}x_i,\qquad
I(\G,\bx)=\sum_{S\in\mcal I(\G)}x_S,\qquad
\supp(\mvec)=\{i\in I:m_i>0\}.
\]
The generalized chromatic polynomial defind as follows (see \cite{BHV}).
\begin{definition}\label{def:multicoloring}
For a nonnegative integer $q$, a proper vertex multi-coloring associated
to $\mvec$ is a map
$\tau_{\G}^{\mvec}:I\longrightarrow\mathcal P(\{1,\ldots,q\})$
such that
\[
|\tau_{\G}^{\mvec}(i)|=m_i,
\qquad
\tau_{\G}^{\mvec}(i)\cap\tau_{\G}^{\mvec}(j)=\varnothing
\quad\text{if }e(i,j)\in E.
\]
Let $P_k(\mvec,\G)$ be the set of ordered $k$-tuples
$(P_1,\ldots,P_k)$ of nonempty independent subsets such that every
vertex $i$ belongs to exactly $m_i$ of the sets $P_j$'s.
The set $P_0(\bzero,\G)$ consists of the empty tuple, and
$P_0(\mvec,\G)=\varnothing$ when $\mvec\neq\bzero$.
\end{definition}

We recall the multicoloring expansion
\cite[Proposition~1]{BHV} and the independence-power identity
\cite[Theorem~7]{BHV}, see 
\cite[Proposition~3.7]{AKV} as well.
\begin{theorem}[Multicoloring identities]\label{thm:multicoloring-identities}
The number of proper vertex multi-colorings is the value at $q$ of
the polynomial
\begin{equation}\label{eq:chromatic-expansion}
\pi_{\G}^{\mvec}(q)
=\sum_{k=0}^{|\mvec|}|P_k(\mvec,\G)|\binom qk.
\end{equation}
Moreover, 
\begin{equation}\label{eq:independence-power}
I(\G,\bx)^q=\sum_{\mvec\geq0}\pi_{\G}^{\mvec}(q)\bx^{\mvec}.
\end{equation}
In particular, $\pi_{\G}^{\bzero}(q)=1$.
\end{theorem}

Grouping the tuples in \eqref{eq:chromatic-expansion} by their
independent-set multiplicities gives the following equivalent
multinomial form of that established identity.

\begin{corollary}[Multinomial form]\label{cor:independent-multinomial}
Let $\mcal I^*(\G)=\mcal I(\G)\setminus\{\varnothing\}$, and write
$(q)_k=q(q-1)\cdots(q-k+1)$, with $(q)_0=1$. Then
\begin{equation}\label{eq:independent-multinomial}
\pi_{\G}^{\mvec}(q)
=\sum_{\substack{n_P\geq0\ (P\in\mcal I^*(\G))\\
\sum\limits_{P: i\in P}n_P=m_i\ (i\in I)}}
\frac{(q)_{\sum_Pn_P}}{\prod_P n_P!}.
\end{equation}
This is a finite expression for an arbitrary graph $\G$.
\end{corollary}

\subsection{The clique blow-up}
Let $\G[\mvec]$ be the clique blow-up of $\G$, obtained by replacing
vertex $i$ by a clique $K_{m_i}$ of size $m_i$ and omitting it when $m_i=0$.
Its vertices are $i^1,\ldots,i^{m_i}$ for each $i\in\supp(\mvec)$. All copies of a
fixed vertex form a clique. Every copy of $i$ is adjacent to every
copy of $j$ when $e(i,j)\in E$, and there are no other edges.
Write $\pi_{\G[\mvec]}(q)$ for its ordinary chromatic polynomial. The following Proposition follows from \cite[Proposition~2]{BHV} and \cite[Theorem~1.1(a)]{Stanley}, also see references therein.
\begin{proposition}[Whitney expansion]\label{prop:blow-up}
 For every $\mvec\geq 0$
chromatic polynomials satisfy
\begin{equation}\label{eq:blow-up}
\pi_{\G}^{\mvec}(q)=\frac{\pi_{\G[\mvec]}(q)}{\mvec!}.
\end{equation}
Whitney's expansion gives the
following formula.
If $H=\G[\mvec]$ and $c_H(F)$ denotes the number of connected components
of the spanning graph $(V(H),F)$, including isolated vertices, then
\begin{equation}\label{eq:edge-subsets}
\pi_{\G}^{\mvec}(q)
=\frac1{\mvec!}\sum_{F\subseteq E(H)}(-1)^{|F|}q^{c_H(F)},
\end{equation}
where for every $\mvec$, $c_H(\varnothing)=|V(H)|=|\mvec|$ and
when $\mvec=\bzero$, $H$ has no vertices, so the only term in
\eqref{eq:edge-subsets} is $q^0=1$.
\end{proposition}
{

We also use multiplicativity over support components, an immediate
consequence of the multi-coloring definition in
\cite[Section~2.3]{BHV}.

\begin{proposition}\label{prop:components}
If the nonempty support $\supp(\mvec)$ has components
$S_1,\ldots,S_s$ in $\G(\supp(\mvec))$, then
\begin{equation}\label{eq:component-factorization}
\pi_{\G}^{\mvec}(q)
=\prod_{a=1}^s\pi_{\G(S_a)}^{\mvec|_{S_a}}(q).
\end{equation}
In particular, $q^s$ divides $\pi_{\G}^{\mvec}(q)$.
\end{proposition}
}
\section{Denominator quotient and its powers}\label{sec:quotient}

\subsection{}
We use the Weyl--Kac character formula for $\mf g_{\mathrm{re}}$ and the
Weyl--Kac--Borcherds denominator identity for $\mf g(A)$; see \cite[Chapter~10]{Kac} and
\cite{Borcherds}. We use Kang's character form of the quotient
\cite[equation~(2.1)]{Kang} and the established graph identities
recalled above, and give only their single-attachment specializations.
Define
\[
D_{\mf g}=\prod_{\xi\in\Delta^+}(1-e^{-\xi})^{\mult(\xi)},
\qquad
D_{\mf g_{\mathrm{re}}}=\prod_{\al\in\Delta^+_{\mathrm{re}}}
(1-e^{-\al})^{\dim(\mf g_{{\mathrm{re}}})_\al}.
\]
and set
\begin{equation}\label{eq:S-product}
S=\frac{D_{\mf g}}{D_{\mf g_{\mathrm{re}}}}
=\prod_{\xi\in\Psi}(1-e^{-\xi})^{\mult(\xi)}.
\end{equation}
For $P\in\mcal I(\G_{\im})$, set $\ga_P=\displaystyle\sum_{j\in P}\be_{j}$. Note that $-\ga_P$ is a dominant integral weight for $\mf g_{\mathrm{re}}.$ Now choose a
dominant integral weight $\la_P$ satisfying:
\[
\la_P(h_{\alpha_i})=(-\ga_P)(h_{\alpha_i})
=c\delta_{i,u}\mathbf1_{\{v\in P\}},  \qquad i\in I^{\re},
\]
where $\mathbf1_{\{v\in P\}} = 1 ,$ if $v \in P$ and zero otherwise, also define
\[
\operatorname{Ch}_{\mf g_{\mathrm{re}}}(-\ga_P)
=e^{-\ga_P}e^{-\la_P}\operatorname{Ch} V_{\mf g_{\mathrm{re}}}(\la_P).
\]
The normalized-character independence recalled in
subsection \ref{subsec: ch of g_X} makes this independent of the chosen
extension of $\la_P$. The following lemma is immediate.

\begin{lemma}\label{lem:attachment-character}
For every $P\in\mcal I(\G_{\im})$,
\begin{equation}\label{eq:shifted-character}
\operatorname{Ch}_{\mf g_{\mathrm{re}}}(-\ga_P)
=x_PH_\la^{\mathbf1_{\{v\in P\}}},
\end{equation}
where $\mathbf1_{\{v\in P\}} = 1 ,$ if $v \in P$ and zero otherwise. \qed
\end{lemma}
\noindent
We recall the following denominator identity for BKM Lie algebra $\mf g$ from \cite{Borcherds}. 
$$\sum\limits_{(\gamma, w) \in \Omega\times W_{I^{\mathrm{re}}}}(-1)^{\mathrm{ht}(\gamma)+\ell(w)}e^{w(\rho-\gamma)-\rho} = \prod\limits_{\alpha\in \Delta^+}(1-e^{-\alpha})^{\mathrm{dim}\mathfrak{g}_\alpha},$$
where $\Omega=\{0\} \cup\{ \ga \in Q^+: \ga= \beta_1+\cdots +\beta_r, \,\, \text{ and } (\be_i,\be_j)=0, \, \text{ for } 1\leq i \neq j \leq r \text{ and } \be_i \text{ are distinct imaginary simple roots} \}.$ It is easy to see that if $\rho'$ is a Weyl vector for $\mf g_{\mathrm{re}}$, then $w(\rho')-\rho'= w(\rho)-\rho,$ for all $w \in W_{I^{\mathrm{re}}}.$   Hence for subalgebra $\mathfrak{g}_{\mathrm{re}}$ of $\mf g$, the denominator identity is given by: 
     $$\sum\limits_{ w\in W_{I^{\mathrm{re}}}}(-1)^{\ell(w)}e^{w(\rho)-\rho}=    \prod\limits_{\alpha\in \Delta^+(\mathfrak{g}_{\mathrm{re}})}(1-e^{-\alpha})^{{(\mathrm{dim}\mathfrak{g}_{\mathrm{re}})}_\alpha},$$
Therefore, we have

\begin{align}\label{main identity}
    S &=\frac{\sum\limits_{(\gamma, w) \in \Omega\times W_{I^{\mathrm{re}}}}(-1)^{\mathrm{ht}(\gamma)+\ell(w)}e^{w(\rho-\gamma)-\rho}}{
\sum\limits_{ w\in W_{I^{\mathrm{re}}}}(-1)^{\ell(w)}e^{w(\rho)-\rho}} \cr
    & =\sum\limits_{\gamma \in \Omega}(-1)^{\mathrm{ht}(\gamma)} \left(\frac{\sum\limits_{w\in W_{I^{\mathrm{re}}}}(-1)^{\ell(w)}e^{w(\rho - \gamma)-(\rho-\gamma)}}{\sum\limits_{w\in W_{I^{\mathrm{re}}}}(-1)^{\ell(w)}e^{w(\rho)-\rho}}\right)e^{-\gamma}
\end{align}
Observe that $-\ga(h_{\alpha_i}) \in \Z_+$ for all $i \in {I^{\mathrm{re}}}$, $\gamma \in \Omega$ and hence $V_{\mf g_{\mathrm{re}}}(-\ga)$ is an irreducible integrable highest weight $\mf g_{\mathrm{re}}$ module. Thus we have 
   $$S= \dis\sum_{\ga \in \Omega} (-1)^{{\rm {ht}}(\gamma)}\operatorname{Ch}_{\mf g_{\mathrm{re}}}V(-\ga) .$$
Moreover, it is easy to observe that each element of $\Omega$ corresponds to a unique element of $\mcal I(\G_{\im})$. Hence we have the following proposition by the use of Lemma \ref{lem:attachment-character}.
\begin{proposition}\label{cor:S}

\begin{equation}\label{eq:S-independence}
{
S=\sum_{P\in\mcal I(\G_{\im})}(-1)^{|P|}
x_PH_\la^{\mathbf1_{\{v\in P\}}}
=I(\G_{\im},\mathbf z),
}
\end{equation}
where $z_v=-x_vH_\la$ and $z_j=-x_j$ for $j\neq v$. In particular, $S-1\in(x_1,\ldots,x_n)\widehat R$.
\end{proposition}
\begin{remark}
    Proposition \ref{cor:S}, can be obtained from [\cite{Kang}, equation (2.1)] by taking $J=I^{\re}$ as the subset of $I^{re}$ and the specialization introduced here.  
\end{remark}
\subsection{}
For $t\in\Z_{\geq0}$, define
\begin{equation}\label{eq:M}
M_\la(\eta;t)
=\dim\left(V_{\mf g_{\mathrm{re}}}(\la)^{\otimes t}\right)_{t\la-\eta}.
\end{equation}
Here $V_{\mf g_{\mathrm{re}}}(\la)^{\otimes0}=\C$ and
$M_\la(\eta;0)=\delta_{\eta,0}$. Further we get that $M_\la(\eta;t)=H_\la^t[e^{-\eta}]$, equality holds because the
character of a tensor product is the product of the characters. In particular,
$M_\la(0;t)=1$ for every $t$.

\begin{corollary}
\label{cor:coefficient}
For $\mu=\eta+\nu(\mvec)$ and an integer indeterminate $q$, the following holds.
\begin{equation}
S^q[\bx^{\mvec}]
=(-1)^{|\mvec|}H_\la^{m_v}\pi_{\G_{\im}}^{\mvec}(q).
\label{eq:imaginary-coefficient}
\end{equation}
\begin{equation}
{S^q[e^{-\mu}]
=(-1)^{|\mvec|}M_\la(\eta;m_v)\pi_{\G_{\im}}^{\mvec}(q).}
\label{eq:coefficient}
\end{equation}
\begin{equation}\label{eq:coefficient-edge-sum}
S^q[e^{-\mu}]
=\frac{(-1)^{|\mvec|}M_\la(\eta;m_v)}{\mvec!}
\sum_{F\subseteq E(\mcal G_{\im}(\mathbf m))}(-1)^{|F|}q^{c_
{\mcal G_{\im}[\mathbf m]}(F)}.
\end{equation}
\end{corollary}
\begin{proof}
By using the identity \eqref{eq:independence-power} in equation (\ref{eq:S-independence}), each monomial becomes
\[
\mathbf z^{\mvec}=(-1)^{|\mvec|}\bx^{\mvec}H_\la^{m_v},
\]
which proves the equation (\ref{eq:imaginary-coefficient}) and then using the fact $M_\la(\eta;t)=H_\la^t[e^{-\eta}]$ we get \eqref{eq:coefficient}. Again substitution of the Whitney expansion \eqref{eq:edge-subsets} in \eqref{eq:coefficient} gives the form \eqref{eq:coefficient-edge-sum}.
\end{proof}

\subsection{The chromatic character} For $\nu(\mathbf m)$ with $\mathbf m \geq 0$, define
\begin{equation}\label{eq:chromatic-character-definition}
\pi_{\G_{\mathrm{im}}}^{\mathrm{ch},\mvec}(q)
=\sum_{k=0}^{|\mvec|}\binom qk
\sum_{(P_1,\ldots,P_k)\in P_k(\mvec,\G_{\mathrm{im}})}
\prod_{r=1}^k\operatorname{Ch}_{\mf g_{\mathrm{re}}}(-\ga_{P_r}).
\end{equation}
The empty tuple contributes $1$ when $\mvec=\bzero$. Note that every tuple in \eqref{eq:chromatic-character-definition} has the
same imaginary monomial $e^{-\nu(\mvec)}$ and the same real factor
$H_\la^{m_v}$. Factoring them out and using
\eqref{eq:chromatic-expansion} we get the following formula.
\begin{equation}\label{eq:chromatic-character}
\pi_{\G_{\mathrm{im}}}^{\mathrm{ch},\mvec}(q)
=e^{-\nu(\mvec)}H_\la^{m_v}\pi_{\G_{\mathrm{im}}}^{\mvec}(q).
\end{equation}

\section{Linear coefficients and root multiplicities}\label{sec:multiplicities}

\subsection{The graph coefficient}
For $\mvec\neq\bzero$, put
\begin{equation}\label{eq:kappa}
\kappa_{\G}(\mvec)=(-1)^{|\mvec|-1}\pi_{\G}^{\mvec}(q)[q].
\end{equation}
We call $\kappa_{\mcal G}(\mvec) $ the graph coefficient of the graph $\mcal G$. Note that Proposition \ref{prop:components} implies, if $\mcal G(\mathbf m)$ is disconnected then $q^2|\pi_{\G}^{\mvec}(q).$ Now the following proposition follows by taking coefficient of $q$ in equations \eqref{eq:chromatic-expansion} and  \eqref{eq:edge-subsets} with the noted fact.

\begin{proposition}\label{prop:kappa-general}
For an arbitrary finite simple graph $\mcal G$ following expressions holds for $\kappa_{\G}(\mvec).$
\begin{align}
\kappa_{\G}(\mvec)
&=\sum_{k=1}^{|\mvec|}\frac{(-1)^{|\mvec|+k}}{k}
|P_k(\mvec,\G)|,\label{eq:kappa-tuples}\\
\kappa_{\G}(\mvec)
&=\frac{(-1)^{|\mvec|-1}}{\mvec!}
\sum_{\substack{F\subseteq E(\G[\mvec])\\
(V(\G[\mvec]),F)\text{ connected}}}(-1)^{|F|}.
\label{eq:kappa-edges}
\end{align}
If $\G(\supp(\mvec))$ is disconnected, then $\kappa_{\G}(\mvec)=0$.
\end{proposition}

\subsection{The logarithm and divisor relation}
For $S=1+U$, using \eqref{eq:formal-power} we have
\begin{equation}\label{eq:logarithm}
S^q[q]=\sum_{k\geq1}\frac{(-1)^{k-1}}kU^k=\log S.
\end{equation}
For $\mu=\eta(\mathbf k)+\nu(\mvec)$ with $\mvec\neq\bzero$, define
\begin{align*}
g=g(\bk,\mvec)
&=\gcd\bigl(\{k_i:i\in {I^{\mathrm{re}}}\}\cup\{m_j:j\in {I^{\mathrm{im}}}\}\bigr),\\
C(\mu)&=-S^q[e^{-\mu}][q]=(-\log S)[e^{-\mu}].
\end{align*}
Therefore by equations \eqref{eq:coefficient} and \eqref{eq:kappa} we have,
\begin{equation}\label{eq:C-factor}
C(\mu)=M_\la(\eta;m_v)\kappa_{\G}(\mvec).
\end{equation}
The product expression \eqref{eq:S-product} gives
\begin{equation}\label{eq:logarithmic-product}
-\log S=\sum_{\xi\in\Psi}\mult(\xi)
\sum_{d\geq1}\frac{e^{-d\xi}}d.
\end{equation}
Now taking coefficient of $e^{-\mu}$ in $-\log S$ and using \eqref{eq:C-factor}, we have 
\begin{equation}\label{eq:divisor-relation}
C(\mu)=\sum_{d\mid g}\frac1d\mult(\mu/d).
\end{equation}

\begin{theorem}\label{cor:root-general}
If $\mvec\neq\bzero$, then
\begin{equation}\label{eq:root-general}
{
\mult\bigl(\eta(\mathbf k)+\nu(\mvec)\bigr)
=\sum_{d\mid g(\bk,\mvec)}\frac{\Mob(d)}d
M_\la\left(\frac\eta d;\frac{m_v}d\right)
\kappa_{\G_{\mathrm{im}}}\left(\frac\mvec d\right),
}
\end{equation}
here $\Mob$ is the M\"obius function. In particular,
disconnected imaginary support implies $\mult(\mu)=0$.
\end{theorem}
\begin{proof}
Proof of \eqref{eq:root-general} follows from \eqref{eq:divisor-relation} by applying M\"obius inversion formula. Further, last statement follows from Proposition \ref{prop:kappa-general}.
\end{proof}
\begin{corollary}
    If $\mvec, \mathbf k \neq 0$ but $m_v =0$. Then $\mult(\eta(\mathbf k)+\nu(\mathbf m))=0.$ 
\end{corollary}
\begin{proof}
    Note that if $m_v=0$ and $\eta=\eta(\mathbf k)\neq0$, then $M_\la(\frac{\eta}{d};0)=0$, by subsection 4.2. Hence the result by Theorem \ref{cor:root-general}.
\end{proof}

\begin{remark}
    If $\mvec=\bzero$, then $S^q[e^{-\eta}]=\delta_{\eta,0}$, since
$S-1 \in (x_1,\dots,x_n)\widehat R$ . In this case root is purely $I^{\mathrm{re}}$-supported and hence root multiplicities are those of $\mf g_{\mathrm{re}}$ and can not be computed from $S$.
\end{remark}

\section{Chordal graphs}\label{sec:chordal}

\subsection{Perfect elimination ordering}
We follow the convention of \cite[Section 2.6]{BHV}. Define the closed
neighborhood of a vertex $r$ of the graph $\mcal G$ by
\[
N_{\G}[r]=\{r\}\cup\{s\in I:e(s,r)\in E\}.
\]
An ordering of $I=\{1,\ldots,n\}$ is a perfect elimination ordering
if the earlier neighborhood
$N_{\G}[r]\cap\{1,\ldots,r-1\}$ induces a clique for every $r$.
A graph is chordal if it has no induced cycle of length
at least four. 
It is well known that (see \cite{Dirac}), \cite[Proposition~3]{BHV}; see also
\cite[Theorem~4.1]{Golumbic}), we have that
\begin{theorem}\label{thm:chordal-peo}
A finite simple graph is chordal if and only if it admits a perfect
elimination ordering as above.
\end{theorem}

\subsection{The product formula}
Fix such a perfect elimination ordering and put
\begin{equation}\label{eq:a-r}
a_r(\mvec)=\sum_{\substack{1\leq s\leq r\\s\in N_{\G}[r]}}m_s
=m_r+\sum_{\substack{s<r\\e(s,r)\in E}}m_s.
\end{equation}
We record the following identities from \cite[Theorem~6 and
equation~(2.1)]{BHV}.
\begin{theorem}\label{thm:chordal-product}
For every $\mvec\geq0$,
\begin{align}
\pi_{\G}^{\mvec}(q)
&=\prod_{r=1}^n\binom{q-a_r(\mvec)+m_r}{m_r},
\label{eq:chordal-pi}\\
\pi_{\G}^{\mvec}(-q)
&=(-1)^{|\mvec|}\prod_{r=1}^n\binom{q-1+a_r(\mvec)}{m_r}.
\label{eq:chordal-negative-pi}
\end{align}
Both are polynomial identities in $q$.
\end{theorem}
Combining \eqref{eq:chordal-pi} and \eqref{eq:chordal-negative-pi}
with Corollary~\ref{cor:coefficient} gives the following
identities.

\begin{corollary}\label{cor:chordal-coefficients}
For a chordal graph $\G_{\im}$ with a perfect elimination ordering the following identities hold.
\begin{equation}\label{eq:chordal-coefficient}
S^q[e^{-(\eta+\nu(\mvec))}]
=(-1)^{|\mvec|}M_\la(\eta;m_v)
\prod_{r=1}^n\binom{q-a_r(\mvec)+m_r}{m_r}.
\end{equation}
In particular, for every positive integer $p$,
\begin{equation}\label{eq:chordal-inverse}
{
S^{-p}[e^{-(\eta+\nu(\mvec))}]
=M_\la(\eta;m_v)
\prod_{r=1}^n\binom{p-1+a_r(\mvec)}{m_r}.
}
\end{equation}

\end{corollary}

\subsection{Connected support and the linear coefficient}
We record an elementary observation about the chosen ordering.
\begin{lemma}\label{lem:starts}
Let $\G$ be chordal with the chosen perfect elimination ordering,
and let $J=\supp(\mvec)\neq\varnothing$. Call $r\in J$ a start if it
has no earlier neighbor in $J$. The number of starts equals the
number of components of $\G(J)$. In particular, if $\G(J)$ is
connected and $r_0=\min J$, then
\[
a_{r_0}(\mvec)=m_{r_0},\qquad
a_r(\mvec)-m_r>0\quad(r\in J\setminus\{r_0\}).
\]
\end{lemma}
\begin{proof}
Let $C$ be a connected component of $\G(J)$. Since $C$ is finite and non-empty, it contains a unique minimum vertex $v^*$ with respect to the perfect elimination ordering. Because $v^*$ is the minimum in $C$, it has no earlier neighbors within $C$. Furthermore, since $C$ is a connected component, $v^*$ has no neighbors in $J \setminus C$, making $v^*$ is a start vertex. This shows every connected component contains at least one start vertex.

Now assume $C$ contains another start vertex $u^* \neq v^*$. Because $C$ is connected, there exists a shortest path $P$ between $u^*$ and $v^*$ in $C$, and $P$ can not be an edge. Let $w$ be the highest-ordered vertex on this path. We first note that $w$ cannot be $v^*$ because $v^*$ is the strictly minimum vertex in $C$. Furthermore, $w$ cannot be $u^*$: if it were, $u^*$'s neighbor on the path would precede $u^*$ in the ordering, contradicting the assumption that $u^*$ is a start vertex. Thus, $w$ is an internal vertex of $P$, and its two immediate neighbors on $P$-say, $x$ and $y$-must both precede $w$ in the ordering ($x < w$ and $y < w$).
\begin{center}
\begin{tikzpicture}[
  vertex/.style={circle, fill=black, inner sep=2pt, minimum size=4pt},
  every label/.style={font=\small},
  thick
]
  \node[vertex, label=below:{$u^*$ }] (u) at (0,0) {};
  \node[vertex, label=below:{$x$}] (x) at (2,0.5) {};
  \node[vertex, label=above:{$w$ }] (w) at (3.5,1) {};
  \node[vertex, label=below:{$y$}] (y) at (5,0.5) {};
  \node[vertex, label=below:{$v^*$ }] (v) at (7,0) {};

  \draw (u) -- node[below] {$\dots$} (x);
  \draw (x) -- (w) -- (y);
  \draw (y) -- node[below] {$\dots$} (v);

  \draw[dashed, blue] (x) -- node[below] {Chord} (y);
\end{tikzpicture}
\end{center}
\noindent By the defining property of a perfect elimination ordering, the earlier neighbours of any vertex form a clique. Therefore, $x$ and $y$ must be adjacent in $\G(J)$. However, this creates a chord in $P$, contradicting the assumption that $P$ is a shortest path. Thus, every connected component of $\G(J)$ contains exactly one start vertex. Consequently, the number of starts equals the number of components of $\G(J)$.

If $\G(J)$ is connected and $r_0 = \min J$. Then, $r_0$ is the unique start vertex in $J$. Thus, every other vertex $r \in J \setminus \{r_0\}$ must have an earlier neighbour in $J$. This forces the strict inequality $a_r(\mvec) - m_r > 0$ for all $r \neq r_0$, while the unique start $r_0$ retains $a_{r_0}(\mvec) = m_{r_0}$.
\end{proof}

Now Theorem \ref{thm:chordal-product}, gives the following immediate consequence for the chordal graph.

\begin{corollary}\label{cor:chordal-kappa}
Suppose $\G$ is chordal with the chosen perfect elimination ordering.
If $\mvec\neq\bzero$ and $J=\supp(\mvec)$ is connected in $\G$, let
$r_0=\min J$. Then
\begin{equation}\label{eq:chordal-kappa}
{
\kappa_{\G}(\mvec)
=\frac1{m_{r_0}}
\prod_{r\in J\setminus\{r_0\}}
\binom{a_r(\mvec)-1}{m_r}.
}
\end{equation}
For disconnected $J$, $\kappa_{\G}(\mvec)=0$.
\end{corollary}
\begin{proof}
From \eqref{eq:chordal-pi}, we have expression of $\pi_{\mvec}^G(q)$,
\[
\pi_{\mvec}^G(q)
=
\binom{q-a_{r_{0}}(\mvec)+m_{r_{0}}}{m_{r_{0}}} \times \cdots \times \binom{q-a_{r_{k-1}}(\mvec) + m_{r_{k-1}}}{m_{r_{k-1}}} 
\]
\noindent Since $a_{r_{0}}(\mvec)=m_{r_{0}}$, the first term is $\binom{q}{m_{r_{0}}}$. For the remaining terms, $a_{r_{i}}(\mvec) - m_{r_{i}}>0$, which means evaluating $\binom{q-(a_{r_{i}}(\mvec) - m_{r_{i}})}{m_{r_{i}}}$ at $q=0$ yields a non-zero constant. We rewrite the first term as $\binom{q}{m_{r_{0}}}=\frac{q}{m_{r_{0}}}\binom{q-1}{m_{r_{0}}-1}$.
So $q$ divides $\pi_{\mvec}^{\G}(q)$ exactly once. Thus, we can write $\pi_{\mvec}^{\G}(q)=a_{1}q+a_{2}q^{2}+\cdots+ a_{n}q^{n}$. By factoring out $q$, we see that $a_{1}$ is the constant term of $(a_{1}+a_{2}q+\cdots+ a_{n}q^{n-1})$. Therefore, to find the linear coefficient $\pi_{\mvec}^{\G}(q)[q] = a_1$, we factor out $q$ and evaluate the remaining expression at $q=0$:
\begin{align*}
[q]\pi_{\mvec}^{\G}(q) 
&= \left. \frac{1}{m_{r_{0}}}\binom{q-1}{m_{r_{0}}-1} \times \cdots \times \binom{q-a_{r_{k-1}}(\mvec) + m_{r_{k-1}}}{m_{r_{k-1}}} \right|_{q=0}\\
&= \frac{1}{m_{r_{0}}}\binom{-1}{m_{r_{0}}-1}\times \cdots \times \binom{-(a_{r_{k-1}}(\mvec)-m_{r_{k-1}})}{m_{r_{k-1}}}\\
&= \frac{1}{m_{r_{0}}} (-1)^{m_{r_0}-1} \binom{m_{r_0}-1}{m_{r_0}-1} \prod_{r_{i}\not=r_{0}} (-1)^{m_{r_i}} \binom{(a_{r_{i}}(\mvec)-m_{r_{i}})+m_{r_{i}}-1}{m_{r_{i}}}\\
&= \frac{1}{m_{r_{0}}} (-1)^{|\mvec|-1}\prod_{r_{i}\not=r_{0}} \binom{a_{r_{i}}(\mvec)-1}{m_{r_{i}}}
\end{align*}
Here we used the identity $\binom{-x}{k}=(-1)^{k}\binom{x+k-1}{k}$, and $(m_{r_0}-1) + \sum_{r_i \neq r_0} m_{r_i} = |\mvec|-1$.
\\
Finally, substituting this into the definition $\kappa_{\G}(\mvec)=(-1)^{|\mvec|-1}\pi_{\G}^{\mvec}(q)[q]$, yields exactly equation (\ref{eq:chordal-kappa}). If $J$ is disconnected, higher power of $q$ divides $\pi_{\mvec}^{\G}$, making $\pi_{\mvec}^{\G}[q]=0$, implying $\kappa_{\G}(\mvec)=0$.
\end{proof}
The order of vanishing of $\pi_{\G}^{\mvec}(q)$ at $q=0$ equals
the number of components of $\G(J)$: this is the standard
ordinary chromatic-polynomial property recalled in
\cite[p.~167]{Stanley}, applied to the clique
blow-up using \eqref{eq:blow-up}. In particular, if $\supp(\mvec)$ is disconnected then $\kappa_{\G}(\mvec)=0$. This fact holds for arbitrary $\G$.

\begin{corollary}[Chordal specialization]
\label{cor:chordal-root}
Suppose $\G_{\mathrm{im}}$ is chordal and $J=\supp(\mvec)\neq\varnothing$ is
connected. In the chosen perfect elimination ordering put $r_0=\min J$.
Then
\begin{equation}\label{eq:chordal-root}
{
\mult\bigl(\eta+\nu(\mvec)\bigr)
=\frac1{m_{r_0}}\sum_{d\mid g(\bk,\mvec)}\Mob(d)
M_\la\left(\frac\eta d;\frac{m_v}d\right)
\prod_{r\in J\setminus\{r_0\}}
\binom{a_r(\mvec)/d-1}{m_r/d}.
}
\end{equation}
For disconnected $J$, the multiplicity is zero.
\end{corollary}
\begin{proof}
This follows by inserting \eqref{eq:chordal-kappa} into
Corollary~\ref{cor:root-general}. The scaling identities
\[
a_r(\mvec/d)=a_r(\mvec)/d,\qquad
d^{-1}(m_{r_0}/d)^{-1}=m_{r_0}^{-1}
\]
give the factor $1/m_{r_0}$.
\end{proof}
\begin{example}\label{ex:diamond}
Let $\G_{\mathrm{im}}$ have vertices $1,2,3,4$ and edges
\[
e(1,2),\ e(1,3),\ e(2,3),\ e(2,4),\ e(3,4).
\]
\begin{center}
\begin{tikzpicture}[
    thick,
    vertex/.style={circle, fill=black, inner sep=1.8pt}
]
  \node[vertex, label=left:$1$]  (v1) at (0,0) {};
  \node[vertex, label=above:$2$] (v2) at (1.5,1) {};
  \node[vertex, label=below:$3$] (v3) at (1.5,-1) {};
  \node[vertex, label=right:$4$] (v4) at (3,0) {};

  \draw (v1) -- (v2);
  \draw (v1) -- (v3);
  \draw (v2) -- (v3);
  \draw (v2) -- (v4);
  \draw (v3) -- (v4);
\end{tikzpicture}
\end{center}
The earlier neighbor sets are $\varnothing$, $\{1\}$, $\{1,2\}$,
and $\{2,3\}$; each is a clique. Thus this ordering is a perfect
elimination ordering, and
\[
a_1=m_1,\quad a_2=m_1+m_2,\quad
a_3=m_1+m_2+m_3,\quad a_4=m_2+m_3+m_4.
\]
Formula \eqref{eq:chordal-pi} becomes
\[
\pi_{\G_{\mathrm{im}}}^{\mvec}(q)
=\binom q{m_1}\binom{q-m_1}{m_2}
\binom{q-m_1-m_2}{m_3}\binom{q-m_2-m_3}{m_4}.
\]
For $\mvec=(1,1,1,1)$ it is $q(q-1)(q-2)^2$, whose linear
coefficient is $-4$. Hence $\kappa_{\G_{\mathrm{im}}}(\mvec)=4$.
The gcd of all coefficients of $\eta+\be_1+\cdots+\be_4$ is $1$,
and $m_v=1$ for every attachment vertex $v$. Therefore by Theorem \ref{cor:root-general},
\[
\mult(\eta+\be_1+\be_2+\be_3+\be_4)=4M_\la(\eta,1).
\]
For $\mvec=(2,1,1,2)$, the values of $a_r$ are $(2,3,4,4)$, so
\[
\kappa_{\G_{\mathrm{im}}}(2,1,1,2)
=\frac12\binom21\binom31\binom32=9.
\]
This vector again has gcd $1$, so for attachment at $v=1$ the root
multiplicity is $9M_\la(\eta;2)$. 
\end{example}

\subsection{Complete graphs}
\begin{corollary}\label{cor:complete}
For $\G_{\mathrm{im}}=K_n$ and $\mvec\geq0$,
\[
\pi_{K_n}^{\mvec}(q)=\frac{(q)_{|\mvec|}}{\mvec!}.
\]
If $\mvec\neq\bzero$, then
\[
\kappa_{K_n}(\mvec)=\frac{(|\mvec|-1)!}{\mvec!},
\qquad
\mult(\eta+\nu(\mvec))
=\sum_{d\mid g}\frac{\Mob(d)}d
M_\la\left(\frac\eta d;\frac{m_v}d\right)
\frac{(|\mvec|/d-1)!}{\prod_{j=1}^n(m_j/d)!}.
\]
\end{corollary}
\begin{proof}
Immediate from the equation \eqref{eq:chordal-pi} and Theorem \ref{cor:root-general}.

\end{proof}

\subsection{Path graphs}
Let $\G_{\mathrm{im}} = L_n$ denote the path with edges $e(r,r+1)$ for $1\leq r<n$. We may attach the real graph with the path at
any $v\in\{1,\ldots,n\}$.
\begin{corollary}\label{cor:path-coefficients}
For every $\mvec\geq0$,
\begin{align}
\pi_{L_n}^{\mvec}(q)
&=\prod_{r=1}^n\binom{q-m_{r-1}}{m_r},
\label{eq:path-pi}\\
S^q[e^{-(\eta+\nu(\mvec))}]
&=(-1)^{|\mvec|}M_\la(\eta;m_v)
\prod_{r=1}^n\binom{q-m_{r-1}}{m_r},
\label{eq:path-coefficient}\\
S^{-p}[e^{-(\eta+\nu(\mvec))}]
&=M_\la(\eta;m_v)
\prod_{r=1}^n\binom{p-1+m_{r-1}+m_r}{m_r}\quad(p\geq1).
\label{eq:path-inverse}
\end{align}
Also
\[
\pi_{L_n}^{\mathrm{ch},\mvec}(q)
=e^{-\nu(\mvec)}H_\la^{m_v}
\prod_{r=1}^n\binom{q-m_{r-1}}{m_r}.
\]
\end{corollary}
\begin{proof}
The polynomial $\pi^{\mathbf m}_{L_n}$ is obtained in \cite[Section~2.7, Example~(1)]{BHV}.
In the natural ordering, $a_r(\mvec)=m_{r-1}+m_r$. Substitution of $\pi^{\mathbf m}_{L_n}$ and $a_r(\mvec)$ 
in \eqref{eq:coefficient}, \eqref{eq:chordal-coefficient}, \eqref{eq:chordal-inverse}, and
\eqref{eq:chromatic-character} gives the remaining formulas.
\end{proof}

\begin{corollary}\label{cor:path-roots}
Suppose $\mvec\neq\bzero$. If $\supp(\mvec)$ is not an interval, then
$\kappa_{L_n}(\mvec)=0$ and $\mult(\eta+\nu(\mvec))=0$.
If $\supp(\mvec)=[r,s]$, then
\begin{align}
\kappa_{L_n}(\mvec)
&=\frac1{m_r}\prod_{j=r+1}^s
\binom{m_{j-1}+m_j-1}{m_j},\label{eq:path-kappa}\\
\mult(\eta+\nu(\mvec))
&=\frac1{m_r}\sum_{d\mid g(\bk,\mvec)}\Mob(d)
M_\la\left(\frac\eta d;\frac{m_v}d\right)
\prod_{j=r+1}^s
\binom{(m_{j-1}+m_j)/d-1}{m_j/d}.
\label{eq:path-root}
\end{align}
\end{corollary}
\begin{proof}
 If $\supp(\mvec)$ is not an interval then $\supp(\mvec)$ is disconnected and hence it follows from Corollary \ref{cor:chordal-kappa} that $\kappa_{L_n}(\mvec)=0$, as every path graph is chordal. Again if $\supp(\mvec)$ is an interval, then the subgraph generated by $\supp(\mvec)$ has a start at $r.$ Hence by Corollary \ref{cor:chordal-kappa} and Corollary \ref{cor:chordal-root} we have the result, since $a_r(\mvec)=m_{r-1}+m_r.$
 
\end{proof}
\subsection{More examples}
The next corollary deals with the root multiplicities of three different kinds of graphs but with special situations.
\begin{corollary}
\label{cor:path-special}
\begin{enumerate}
    \item If there is only one imaginary vertex $\be_1$ attached to $\mathcal{G}_{\mathrm{re}}$, then for $t>0$,
\begin{equation}\label{eq:single-vertex}
\mult(\eta+t\be_1)
=\frac1t\sum_{d\mid\gcd(\{k_i\}_{i \in I^{\re}},t)}\Mob(d)
M_\la(\eta/d;t/d).
\end{equation}
\item Let the graph $\G_{\mathrm{im}}$ is a path graph but the interval $[r,s]$ is squarefree, meaning all $m_i=1$ for all $i \in [r,s].$ Then 
\[
\mult\left(\eta+\sum_{j=r}^s\be_{j}\right)
=\begin{cases}
M_\la(\eta;1),&v\in[r,s],\\
\delta_{\eta,0},&v\notin[r,s].
\end{cases}
\] 
\item For an arbitrary $\G_{\mathrm{im}}$ with any imaginary vertex $j$, 
$\mult(t\be_{j})=1$ if $t=1$ and $0$ if $t>1$.
\end{enumerate}

\end{corollary}
\begin{proof}
The first two assertions are direct evaluations of
\eqref{eq:path-root}: the product is empty for a single vertex
and is $1$ for a squarefree interval.
For an arbitrary $\G_{\im}$, let $\mathbf e_j$ be the $j$th unit vector.
The blow-up $\G_{\im}[t\mathbf e_j]$ is $K_t$, so by Theorem \ref{cor:complete},
$\pi_{\G_{\im}}^{t\mathbf e_j}(q)=\binom qt$ and
$\kappa_{\G_{\im}}(t\mathbf e_j)=1/t$.
Since $M_\la(0;r)=1$, for all $r \geq 0$, Theorem~\ref{cor:root-general} gives
\[
\mult(t\be_{j})=\frac1t\sum_{d\mid t}\Mob(d),
\]
which is $1$ for $t=1$ and $0$ for $t>1$.

\end{proof}

\section{Finite formulas for the real character factor}\label{sec:real-factor}

\subsection{A weight-multinomial expression}
For $\eta\in Q_{\mathrm{re}}^+$ define
\[
\Gamma_\eta=\{\ga\in Q_{\mathrm{re}}^+:\eta-\ga\in Q_{\mathrm{re}}^+,
\ h_\la(\ga)\neq0\}.
\]
If $\eta=\dis\sum_{i \in I^{\re}} k_i\al_i$, every simple-root coefficient of a
member of $\Gamma_\eta$ lies between $0$ and $k_i$, so this set is
finite and it contains $0$.
\begin{proposition}\label{prop:weight-multinomial}
For $\eta\in Q_{\mathrm{re}}^+$ and $t\in\Z_{\geq0}$,
\begin{equation}\label{eq:weight-multinomial}
M_\la(\eta;t)
=\sum_{\substack{n_\ga\geq0\ (\ga\in\Gamma_\eta)\\
\sum_\ga n_\ga=t,\quad \sum_\ga n_\ga\ga=\eta}}
\frac{t!}{\prod_\ga n_\ga!}
\prod_\ga h_\la(\ga)^{n_\ga}.
\end{equation}
\end{proposition}
\begin{proof}

Note that 
\begin{align*}
    M_\la(\eta;t)=H_\la^t[e^{-\eta}] =\left(\dis{\sum_{\ga \in Q_{\mathrm{re}}^+}}h_\la(\ga)e^{-\ga} \right)^t[e^{-\eta}] \cr
   \end{align*}
It is easy to observe that terms with $\gamma\notin\Gamma_\eta$ cannot contribute to the coefficient
of $e^{-\eta}$ in $H_\lambda^t$. Hence
\[
M_\lambda(\eta;t)=\left(
\sum_{\gamma\in\Gamma_\eta}
h_\lambda(\gamma)e^{-\gamma}
\right)^t[e^{-\eta}].
\]
Now applying the multinomial theorem, we get
\begin{align*}
 M_\la(\eta;t)&=
\sum_{\substack{
n_\gamma\geq 0\;(\gamma\in\Gamma_\eta)\\
\sum_{\gamma\in\Gamma_\eta}n_\gamma=t
}}
\frac{t!}
{\prod_{\gamma\in\Gamma_\eta}n_\gamma!}
\prod_{\gamma\in\Gamma_\eta}
\left(
h_\lambda(\gamma)e^{-\gamma}
\right)^{n_\gamma}[e^{-\eta}]
\\
&=
\sum_{\substack{
n_\gamma\geq 0\;(\gamma\in\Gamma_\eta)\\
\sum_{\gamma\in\Gamma_\eta}n_\gamma=t
}}
\frac{t!}
{\prod_{\gamma\in\Gamma_\eta}n_\gamma!}
\left(
\prod_{\gamma\in\Gamma_\eta}
h_\lambda(\gamma)^{n_\gamma}
\right)
e^{-\sum_{\gamma\in\Gamma_\eta}n_\gamma\gamma}[e^{-\eta}] \cr
&=\sum_{\substack{n_\ga\geq0\ (\ga\in\Gamma_\eta)\\
\sum_\ga n_\ga=t,\quad \sum_\ga n_\ga\ga=\eta}}
\frac{t!}{\prod_\ga n_\ga!}
\prod_\ga h_\la(\ga)^{n_\ga}.
\end{align*}

\end{proof}

\subsection{A finite expression using the Weyl group}
Choose a Weyl vector $\rho'$ satisfying $\rho'(h_{\alpha_i})=1$ for every
$i\in {I^{\mathrm{re}}}$, and set
\[
\delta_\la(w)=\la+\rho'-w(\la+\rho'),
\qquad \delta_0(w)=\rho'-w\rho'.
\]
The normalized Weyl--Kac formula \cite[Chapter~10]{Kac} is
\begin{equation}\label{eq:Weyl-ratio}
H_\la=\frac{N_\la}{N_0},\qquad
N_\la=\sum_{w\in W_{I^{\mathrm{re}}}}(-1)^{\ell(w)}e^{-\delta_\la(w)},
\qquad
N_0=\sum_{w\in W_{I^{\mathrm{re}}}}(-1)^{\ell(w)}e^{-\delta_0(w)}.
\end{equation}
Its denominator is also $D_{\mf g_{\mathrm{re}}}$.
\begin{lemma}\label{lem:height-bound}
For every $w\in W_{I^{\mathrm{re}}}$, both $\delta_\la(w)$ and $\delta_0(w)$ lie
in $Q_{\mathrm{re}}^+$, and
\[
\htop\delta_\la(w)\geq\ell(w),\qquad
\htop\delta_0(w)\geq\ell(w).
\]
\end{lemma}
\begin{proof}
For a reduced expression $w=s_{i_1}\cdots s_{i_l}$, it is easy to see that,
\[
\theta-w\theta=\sum_{a=1}^l\theta(h_{i_a})
s_{i_1}\cdots s_{i_{a-1}}(\al_{i_a}),
\]
for $\theta=\rho'$ or $\rho'+\la$. Now it is a standard fact from
\cite[Chapter~3]{Kac} that $s_{i_1}\cdots s_{i_{a-1}}(\al_{i_a})$ is a positive real root of $\mf g_{\mathrm{re}}$ and all these vectors have positive height. Again $\theta$ is dominant integral and hence $\theta(h_{i_a}) \geq 1$ for all $1 \leq a \leq l$. Therefore $\htop (\theta -w \theta)\geq l,$ this completes the proof.
\end{proof}

Expanding the normalized Weyl--Kac character formula
\cite[Chapter~10]{Kac} yields the following finite computational formula for $M_\la(\eta,t)$.

\begin{proposition}
\label{prop:finite-Weyl}
Let $\eta\in Q_{\mathrm{re}}^+$, $t\in\Z_{\geq0}$, $N=\htop(\eta)$, and
$\mcal W_N=\{w\in W_{I^{\mathrm{re}}}:1\leq\ell(w)\leq N\}$.
For nonnegative integer families $(p_w),(r_w)$ indexed by
$\mcal W_N$, put $P=\sum_wp_w$ and $R=\sum_wr_w$. Then
\begin{equation}\label{eq:finite-Weyl}
M_\la(\eta;t)
=\sum_{\substack{p_w,r_w\geq0\\
\sum_wp_w\delta_\la(w)+\sum_wr_w\delta_0(w)=\eta}}
(-1)^{\sum_w\ell(w)(p_w+r_w)}
\frac{(t)_P(-t)_R}{\prod_wp_w!\prod_wr_w!}.
\end{equation}
Here $(z)_a=z(z-1)\cdots(z-a+1)$ is a falling factorial,
including $(z)_0=1$. This sum is finite, including when $W_{I^{\mathrm{re}}}$ is infinite.
\end{proposition}

\begin{proof}
Note that $M_{\lambda}(\eta;t)
=
H_{\lambda}^{\,t}[e^{-\eta}] = N_{\lambda}^{\,t}N_0^{-t}[e^{-\eta}].$
 Since $\delta_{\lambda}(1)=\delta_0(1)=0,$ we may write
\[
N_{\lambda}=1+V,
\qquad
N_0=1+U,
\]
where
\[
V=
\sum_{\substack{w\in W_{I^{\mathrm{re}}}\\w\neq 1}}
(-1)^{\ell(w)}e^{-\delta_{\lambda}(w)} \text{and  }
U=
\sum_{\substack{w\in W_{I^{\mathrm{re}}}\\w\neq 1}}
(-1)^{\ell(w)}e^{-\delta_0(w)}.
\]
Therefore
\[
H_{\lambda}^{\,t}
=
(1+V)^t(1+U)^{-t}.
\]
By Lemma \ref{lem:height-bound}, $\operatorname{ht}\delta_{\lambda}(w)\geq \ell(w)$ and $
\operatorname{ht}\delta_0(w)\geq \ell(w)$.
Consequently, if $\ell(w)>N$, then
\[
\operatorname{ht}\delta_{\lambda}(w)>N,
\qquad
\operatorname{ht}\delta_0(w)>N.
\]
Since all $\delta_{\lambda}(w)$ and $\delta_0(w)$ belong to $Q_{\mathrm{re}}^+$,
such terms cannot contribute to the coefficient of $e^{-\eta}$.
Hence, for computing coefficient of $e^{-\eta}$, we are restricted
to the finite set
\[
W_N
=
\{\,w\in W_{I^{\mathrm{re}}}:1\leq \ell(w)\leq N\,\}.
\]
Thus we may replace $V$ and $U$ by
\[
V_N=
\sum_{w\in W_N}
(-1)^{\ell(w)}e^{-\delta_{\lambda}(w)},
\qquad
U_N=
\sum_{w\in W_N}
(-1)^{\ell(w)}e^{-\delta_0(w)}.
\]

We now expand
\[
(1+V_N)^t
=
\sum_{P\geq 0}\binom{t}{P}V_N^P.
\]
For a family of nonnegative integers $(p_w)_{w\in W_N}$ satisfying $P=\sum_{w\in W_N}p_w,$ the multinomial theorem gives
\[
V_N^P
=
\sum_{\substack{p_w\geq 0\\\sum_w p_w=P}}
\frac{P!}{\prod_w p_w!}
(-1)^{\sum_w\ell(w)p_w}
e^{-\sum_w p_w\delta_{\lambda}(w)}.
\]
Hence we get
\[
(1+V_N)^t
=
\sum_{(p_w)}
(-1)^{\sum_w\ell(w)p_w}
\frac{(t)_P}{\prod_w p_w!}
e^{-\sum_w p_w\delta_{\lambda}(w)},
\]
since $\binom{t}{P}=\frac{(t)_P}{P!}.$
Similarly, if $(r_w)_{w\in W_N}$ is a family of nonnegative integers with $R=\sum_{w\in W_N}r_w,$ then we get
\[
(1+U_N)^{-t}
=
\sum_{(r_w)}
(-1)^{\sum_w\ell(w)r_w}
\frac{(-t)_R}{\prod_w r_w!}
e^{-\sum_w r_w\delta_0(w)}.
\]
Therefore we have,
\begin{align*}
M_{\lambda}(\eta;t) &=
H_{\lambda}^{\,t}[e^{-\eta}] \cr
&=\sum_{(p_w),(r_w)}
(-1)^{
\sum_w\ell(w)(p_w+r_w)}
\frac{(t)_P(-t)_R}
{\prod_w p_w!\prod_w r_w!}
e^{-
\left(
\sum_w p_w\delta_{\lambda}(w)
+
\sum_w r_w\delta_0(w)
\right)} [e^{-\eta}]\cr
&=\sum_{\substack{p_w,r_w\geq0\\
\sum_w p_w\delta_{\lambda}(w)
+
\sum_w r_w\delta_0(w)=\eta}}
(-1)^{
\sum_w\ell(w)(p_w+r_w)}
\frac{(t)_P(-t)_R}
{\prod_w p_w!\prod_w r_w!},
\end{align*}
where $P=\sum_w p_w,$  $R=\sum_w r_w$ and all sums are taken over $w\in W_N$.\\

\noindent It remains only to verify finiteness. If a family
$(p_w),(r_w)$ contributes, then
\[
\sum_w p_w\delta_{\lambda}(w)
+
\sum_w r_w\delta_0(w)
=
\eta.
\]
Taking heights and using Lemma \ref{lem:height-bound}, we obtain
\[
N
=
\operatorname{ht}(\eta)
\geq
\sum_w p_w+\sum_w r_w
=
P+R.
\]
Hence
\[
P+R\leq N.
\]
Since $W_N$ is finite, only finitely many families
$(p_w),(r_w)$ can occur. This completes prove.

\end{proof}

\begin{remark}
Equations \eqref{eq:independent-multinomial} and
\eqref{eq:finite-Weyl}, substituted into \eqref{eq:coefficient},
give a finite expression for every coefficient from the graph and
Cartan data alone. Equations \eqref{eq:kappa-tuples} and
\eqref{eq:root-general} then give a finite root-multiplicity expression.

\end{remark}

\section{Type \texorpdfstring{$A$}{A} with an arbitrary attached graph}\label{sec:type-A}

\subsection{Attachment at any real vertex and \texorpdfstring{$c=1$}{c=1}}
Let $\mathcal{G}_{\mathrm{re}}=A_l$, so $\mf g_{\mathrm{re}}=\mf{sl}_{l+1}$, with
\[
\al_i=\eps_i-\eps_{i+1},\qquad
\omega_u=\sum_{j=1}^u\eps_j,  
\]
where $ 1\leq i \leq l$.  Note that here $u$ is a vertex leveled by an integer $1,2,\dots, l,$ i.e, $1 \leq u \leq l.$
The graph $\G_{\mathrm{im}}$ still has $n$ imaginary vertices and is arbitrary.
Assume $c=1$, so the real highest weight is $\la= c \Lambda_u=\omega_u$, see subsection \ref{subsec: ch of g_X}.
Let
\[
\mcal T_u=\{T\subseteq\{1,\ldots,\ell+1\}:|T|=u\},
\qquad
b_i(T)=\min(i,u)-|T\cap\{1,\ldots,i\}|, \, 1\leq i \leq l.
\]
Also denote $b(T)=(b_1(T),\dots,b_\ell(T))$ for all $T \in \mcal T_u.$
The realization $$V_{\mf g_{\mathrm{re}}}(\omega_u)=\bigwedge^u\C^{\ell+1}$$ and its
weight character are standard; see \cite[Section~15.2, p.~221]{FH}.
\begin{proposition}
\label{prop:A-character}
For $\eta=\sum_{i=1}^{\ell}k_i\al_i$ and $t\in\Z_{\geq0}$,
the normalized character and its tensor-power coefficients are
\begin{align}
H_{\omega_u}(\by)
&=\sum_{T\in\mcal T_u}\mathbf y^{b(T)},
\label{eq:A-character}\\
M_{\omega_u}(\eta;t)
&=\sum_{\substack{n_T\geq0\ (T\in\mcal T_u)\\
\sum_Tn_T=t\\
\sum_Tn_Tb_i(T)=k_i\ (1\leq i\leq\ell)}}
\frac{t!}{\prod_Tn_T!}.
\label{eq:A-M}
\end{align}
\end{proposition}

\begin{proof}
For $T=\{j_1<\cdots<j_u\}\in \mathcal{T}_u,$ let $e_T=e_{j_1}\wedge\cdots\wedge e_{j_u}.$  Then $e_T$ is a weight vector of weight $
\operatorname{wt}(e_T)
=\dis\sum_{j\in T}\varepsilon_j.$ The highest-weight vector $
e_1\wedge\cdots\wedge e_u,$
has weight $\omega_u.$
Therefore
\begin{equation}
\omega_u-\operatorname{wt}(e_T)
=
\sum_{j=1}^u\varepsilon_{j}
-
\sum_{j\in T}\varepsilon_{j}.
\label{eq:weight-difference}
\end{equation}
Now we express the right-hand side of \eqref{eq:weight-difference} in terms of the simple roots.
Suppose
\[
\gamma=\sum_{r=1}^{\ell+1}a_r\varepsilon_r
      =\sum_{i=1}^{\ell}c_i\alpha_i.
\]
Then we have
\[
c_i=a_1+\cdots+a_i,
\qquad 1\leq i\leq \ell.
\]
With this observation it is easy to see that the coefficient of $\alpha_i$ in
$\omega_u-\operatorname{wt}(e_T)$ is $b_i(T)$.
Thus
\begin{equation}
\omega_u-\operatorname{wt}(e_T)
=
\sum_{i=1}^{\ell}b_i(T)\alpha_i.
\label{eq:b-coordinates}
\end{equation}
Since the weight spaces of $V_{\mf g_{\mathrm{re}}}(\omega_u)$ have multiplicity one and weights are of the form $e_T$, indexed by
$T\in\mathcal{T}_u$. Hence we have
\[
\operatorname{Ch}V_{\mf g_{\mathrm{re}}}(\omega_u)
=
\sum_{T\in\mathcal{T}_u}
e^{\operatorname{wt}(e_T)}.
\]
Therefore, using \eqref{eq:b-coordinates} we have
\begin{align*}
H_{\omega_u}
&=
e^{-\omega_u}\operatorname{Ch}V_{\mf g_{\mathrm{re}}}(\omega_u)\\
&=\sum_{T\in\mathcal{T}_u}
e^{\operatorname{wt}(e_T)-\omega_u}
\\
&=
\sum_{T\in\mathcal{T}_u}
e^{-\sum_{i=1}^{\ell}b_i(T)\alpha_i} \cr
&=\sum_{T\in\mathcal{T}_u}
\prod_{i=1}^{\ell}y_i^{b_i(T)} =\sum_{T\in\mathcal{T}_u}
\mathbf y^{b(T)}.
\end{align*}
This proves \emph{(9.1)}.
Applying the multinomial theorem we have
\[
H_{\omega_u}^t
=
\left(
\sum_{T\in\mathcal{T}_u}
\prod_{i=1}^{\ell}y_i^{b_i(T)}
\right)^t=
\sum_{\substack{
n_T\geq 0\\
\sum_{T\in\mathcal{T}_u}n_T=t
}}
\frac{t!}{\prod_{T\in\mathcal{T}_u}n_T!}
\prod_{i=1}^{\ell}
y_i^{\sum_{T\in\mathcal{T}_u}n_Tb_i(T)}.
\]
Since $e^{-\eta}=
\prod_{i=1}^{\ell}y_i^{k_i},$ a term contributes to the coefficient of $e^{-\eta}$ precisely when
\[
\sum_{T\in\mathcal{T}_u}n_Tb_i(T)=k_i,
\qquad 1\leq i\leq \ell.
\]
Therefore
\[
M_{\omega_u}(\eta;t)
=
H_{\omega_u}^t[e^{-\eta}]
=\sum_{\substack{
n_T\geq 0\;(T\in\mathcal{T}_u)\\
\sum_T n_T=t\\
\sum_T n_Tb_i(T)=k_i,\;1\leq i\leq \ell
}}
\frac{t!}{\prod_T n_T!}.
\]
This proves \emph{(9.2)} and hence the proposition.

\end{proof}

\begin{corollary}\label{cor:A-general-vertex}
For any imaginary graph $\G_{\im}$, any attachment vertex $v$ in it, and
any real vertex $u$ of $A_l$, with $c=1$,
\begin{equation}\label{eq:A-general-vertex}
S^q[e^{-(\eta+\nu(\mvec))}]
=(-1)^{|\mvec|}\pi_{\G}^{\mvec}(q)
\sum_{\substack{n_T\geq0\\
\sum_Tn_T=m_v\\
\sum_Tn_Tb_i(T)=k_i\ (1\leq i\leq\ell)}}
\frac{m_v!}{\prod_Tn_T!}.
\end{equation}
The root multiplicities are obtained by using \eqref{eq:A-M} in
\eqref{eq:root-general}; if $\G$ is chordal, use
\eqref{eq:chordal-root} instead.
\end{corollary}
\begin{proof}
Substitute the standard character computation \eqref{eq:A-M}
into Corollary ~\ref{cor:coefficient} and Theorem~\ref{cor:root-general}.
\end{proof}

\subsection{The real endpoint specialization}
Take $u=\ell$ and still $c=1$. Define
\begin{equation}\label{eq:K}
K_\ell(\bk;t)=
\begin{cases}
\displaystyle
\frac{t!}{k_1!\,\prod_{i=2}^\ell(k_i-k_{i-1})!\,(t-k_\ell)!},
&0\leq k_1\leq\cdots\leq k_\ell\leq t,\\[3mm]
0,&\text{otherwise}.
\end{cases}
\end{equation}

\begin{proposition}
\label{prop:A-endpoint}
For $u=\ell$ and $c=1$,
\begin{equation}\label{eq:A-endpoint-H}
H_{\omega_\ell}(\by)
=1+y_\ell+y_{\ell-1}y_\ell+\cdots+y_1\cdots y_\ell,
\qquad
M_{\omega_\ell}(\eta;t)=K_\ell(\bk;t).
\end{equation}
\end{proposition}
\begin{proof}
We specialize Proposition \ref{prop:A-character} to the endpoint $u=\ell$. Note that every $T\in\mathcal{T}_{\ell}$ is obtained by deleting exactly one
element from $\{1,\ldots,\ell+1\}$. For
$1\leq r\leq \ell+1$, write
\[
T_r
=
\{1,\ldots,\ell+1\}\setminus\{r\}.
\]
By Proposition~\ref{prop:A-character},
\[
H_{\omega_\ell}(y)
=
\sum_{T\in\mathcal{T}_{\ell}}
\prod_{i=1}^{\ell}y_i^{b_i(T)}, \text{ where } b_i(T)= \min(i,\ell)-|T\cap\{1,\ldots,i\}|.
\]

\noindent Since $1\leq i\leq\ell$, we have $\min(i,\ell)=i$, and hence $b_i(T)=i-|T\cap\{1,\ldots,i\}|.$ From this it clearly follows that $b_i(T_{l+1})=0$ for all $i$. Further, it is easy to see that for $1 \leq i,r \leq l,$
 $$b_i(T_r)= \begin{cases}
     0, \,\,\,\, \text{ if }, \,\, i <r\\
     1,   \,\,\,\, \text{ if }, \,\, i \geq r.
 \end{cases}$$
Consequently, we get
\[
H_{\omega_\ell}(y)
=
1+y_{\ell}
+y_{\ell-1}y_{\ell}
+\cdots
+y_1y_2\cdots y_{\ell}.
\]
This proves the first assertion. Now by definition,
\[
M_{\omega_\ell}(\eta;t)
=
H_{\omega_\ell}(y)^t
\left[ \mathbf y^ \mathbf k\right].
\]
Let $c_0$ denote the number of times the term $1$ is chosen in the
multinomial expansion, and for $1\leq r\leq\ell$, let $c_r$ denote
the number of times the term $y_r y_{r+1}\cdots y_{\ell}$
is chosen. Then we have the following relations:
\[
\dis\sum_{j=0}^{\ell}c_j=t, \,\,\,\,\,  \text{ and} \,\,\,\,\,  k_r=  \dis\sum_{j=1}^{r}c_j, \,\,\,\, 1 \leq r \leq l.
\]
From this two relations we can conclude that a term contribute to the coefficient of $\mathbf y^{\mathbf k}$ only when $0 \leq k_1 \leq k_2 \leq \dots \leq k_l\leq t$ and in this case the coefficient is 
\[
\frac{t!}{c_0!c_1!\cdots c_{\ell}!}=\frac{t!}{
k_1!
\prod_{i=2}^{\ell}(k_i-k_{i-1})!
(t-k_{\ell})!
}.
\]
Therefore $M_{\omega_\ell}(\eta;t)
=K_\ell(\mathbf k;t).$

\end{proof}
\noindent The following corollary follows from Corollaries~\ref{cor:coefficient}, ~\ref{cor:chordal-root}, Proposition \ref{prop:A-endpoint} and Theorem \ref{cor:root-general}.
\begin{corollary}[Type $A$ endpoint specialization]
\label{cor:A-general-G}
For $u=\ell$ and $c=1$,
\begin{equation}\label{eq:A-general-G-coefficient}
{
S^q[e^{-(\eta+\nu(\mvec))}]
=(-1)^{|\mvec|}K_\ell(\bk;m_v)\pi_{\G_{\mathrm{im}}}^{\mvec}(q).
}
\end{equation}
If $\mvec\neq\bzero$, then
\begin{equation}\label{eq:A-general-G-root}
\mult(\eta+\nu(\mvec))
=\sum_{d\mid g(\bk,\mvec)}\frac{\Mob(d)}d
K_\ell\left(\frac\bk d;\frac{m_v}d\right)
\kappa_{\G_{\mathrm{im}}}\left(\frac\mvec d\right).
\end{equation}
If $\G_{\mathrm{im}}$ is chordal and $J=\supp(\mvec)$ is connected, then we get
\begin{equation}\label{eq:A-chordal-root}
\mult(\eta+\nu(\mvec))
=\frac1{m_{r_0}}\sum_{d\mid g(\bk,\mvec)}\Mob(d)
K_\ell\left(\frac\bk d;\frac{m_v}d\right)
\prod_{r\in {I^{\mathrm{re}}}\setminus\{r_0\}}
\binom{a_r(\mvec)/d-1}{m_r/d},
\end{equation}
where $r_0=\min J$ in the perfect elimination ordering.
\end{corollary}

\subsection{The real--imaginary path as a consequence}
Now set $\G=L_{\ell+n}$ and $v=1$. The entire graph is a path with
real vertices $\al_1,\ldots,\al_\ell$ followed by imaginary
vertices $\be_1,\ldots,\be_n$. Taking both entries on every
imaginary edge and the attachment to be $-1$ gives the usual
simply-laced real--imaginary path matrix. The formulas depend only
on the already specified real-row entry and the imaginary zero pattern.

\begin{corollary}[The explicit real--imaginary path formula]
\label{cor:A-path}
With $u=\ell$, $c=1$, $\G_{\mathrm{im}}=L_n$, and $v=1$,
\begin{equation}\label{eq:A-path-coefficient}
S^q[e^{-(\eta+\nu(\mvec))}]
=(-1)^{|\mvec|}K_\ell(\bk;m_1)
\binom q{m_1}\prod_{j=2}^n\binom{q-m_{j-1}}{m_j}.
\end{equation}
For the root-multiplicity assertions, assume $\mvec\neq\bzero$.
If $\supp(\mvec)=[r,s]$, then
\begin{equation}\label{eq:A-path-root}
\mult(\eta+\nu(\mvec))
=\frac1{m_r}\sum_{d\mid g(\bk,\mvec)}\Mob(d)
K_\ell\left(\frac\bk d;\frac{m_1}d\right)
\prod_{j=r+1}^s
\binom{(m_{j-1}+m_{j})/d-1}{m_{j}/d}.
\end{equation}
If the imaginary support is
disconnected, the multiplicity is zero.
\end{corollary}
\begin{proof}
Equation \eqref{eq:A-path-coefficient}  follows from \eqref{eq:path-pi} and
\eqref{eq:A-general-G-coefficient}. Equation \eqref{eq:A-path-root} follows from equations \eqref{eq:A-general-G-root} and \eqref{eq:path-kappa}. 
\end{proof}

\subsection{An arbitrary attachment integer at the real endpoint}
Return to an arbitrary imaginary graph $\G_{\mathrm{im}}$, and allow $c\geq1$
with $u=\ell$ and $\mf g_{\mathrm{re}} \cong \mf{sl}_{\ell+1}$ Define
\[
\mcal B_c=\{\mathbf b=(b_1,\ldots,b_\ell)\in\Z_{\geq0}^\ell:
0\leq b_1\leq\cdots\leq b_\ell\leq c\}.
\]
The standard Schur-module description \cite[Proposition~15.15]{FH},
followed by duality, gives the realization
\[
V_{\mf g_{\mathrm{re}}}(c\omega_\ell)=\Sym^c((\C^{\ell+1})^*).
\]

\noindent It is easy to verify the following lemma.

\begin{lemma}\label{lem:symmetric-monomial-bijection}
There is a bijection
\[
\left\{
(t_0,t_1,\dots,t_\ell)\in \mathbb Z_{\geq 0}^{\ell+1}
\;\middle|\;
t_0+t_1+\cdots+t_\ell=c
\right\}
\longrightarrow \mcal B_c
\]
given by
\[
(t_0,t_1,\dots,t_\ell)
\longmapsto
(b_1,\dots,b_\ell),
\qquad
b_i=t_1+\cdots+t_i
\quad (1\leq i\leq \ell).
\]
Its inverse is
\[
t_1=b_1,\qquad
t_i=b_i-b_{i-1}\quad (2\leq i\leq \ell),
\qquad
t_0=c-b_\ell.
\]
\end{lemma}

\begin{proposition}
\label{prop:A-endpoint-c}
For every $c\geq1$ and $t\in\Z_{\geq0}$,
\begin{align}
H_{c\omega_\ell}(\by)
&=\sum_{\mathbf b\in\mcal B_c}\by^{\mathbf b},
\label{eq:A-endpoint-c-H}\\
M_{c\omega_\ell}(\eta;t)
&=\sum_{\substack{n_{\mathbf b}\geq0\ (\mathbf b\in\mcal B_c)\\
\sum_{\mathbf b}n_{\mathbf b}=t,\quad
\sum_{\mathbf b}n_{\mathbf b}\mathbf b=\bk}}
\frac{t!}{\prod_{\mathbf b}n_{\mathbf b}!}.
\label{eq:A-endpoint-c-M}
\end{align}
\end{proposition}
\begin{proof}

Let $e_1^*,\dots,e_{\ell+1}^*$
be the standard basis of $(\mathbb C^{\ell+1})^*$. Then a basis of
$\operatorname{Sym}^c((\mathbb C^{\ell+1})^*)$ is given by the
symmetric monomials
\[
(e_1^*)^{t_1}\cdots (e_\ell^*)^{t_\ell}(e_{\ell+1}^*)^{t_0},
\qquad
t_0+t_1+\cdots+t_\ell=c, \,\,\, t_i\in\mathbb Z_{\geq 0}.
\]
Let $\mu$ be the weight of such a monomial. Since the weight of
$e_i^*$ is $-\varepsilon_i$, we have
\[
\mu
=
-\sum_{i=1}^{\ell}t_i\varepsilon_i
-t_0\varepsilon_{\ell+1}.
\]
The highest weight of
$\operatorname{Sym}^c((\mathbb C^{\ell+1})^*)$ is $c\omega_\ell$.
Therefore
\begin{align*}
c\omega_\ell-\mu
&=
-c\varepsilon_{\ell+1}
+\sum_{i=1}^{\ell}t_i\varepsilon_i
+t_0\varepsilon_{\ell+1}\\
&=
\sum_{i=1}^{\ell}t_i
(\varepsilon_i-\varepsilon_{\ell+1}), \,\, \text{since } c-t_0=t_1+\cdots+t_\ell\\
&=\sum_{i=1}^{\ell}
t_i(\alpha_i+\cdots+\alpha_\ell)\\
&=
t_1\alpha_1
+(t_1+t_2)\alpha_2
+\cdots+
(t_1+\cdots+t_\ell)\alpha_\ell\\
&=b_1\alpha_1+\cdots+b_\ell\alpha_\ell,
\end{align*}
where $ b_i=t_1+\cdots+t_i,$ for all $1\leq i\leq\ell.$ Hence 
$e^{\mu-c\omega_\ell}
=
y_1^{b_1}\cdots y_\ell^{b_\ell}
=y^{\mathbf b}.$ Hence by Lemma~\ref{lem:symmetric-monomial-bijection} we have
\[
H_{c\omega_\ell}(y)
=
e^{-c\omega_\ell}
\operatorname{ch}V_{A_\ell}(c\omega_\ell)
=
\sum_{\mathbf b\in B_c}y^{\mathbf b}.
\]
This proves \eqref{eq:A-endpoint-c-H}. Now equation  \eqref{eq:A-endpoint-c-M} follows parallelly to the proof of equation  \eqref{eq:A-M}.

\end{proof}
These real factors can be used in all the general-graph, chordal,
and path formulas above.

\subsection{Explicit computations}

\begin{example}\label{ex:A3}
Take $\mathcal{G}_{\mathrm{re}}=A_3$, $u=2$, $c=1$, $\G_{\mathrm{im}}=L_2$, and $v=1$.
The two-element subsets of $\{1,2,3,4\}$ give
\[
\begin{array}{c|cccccc}
T&\{1,2\}&\{1,3\}&\{1,4\}&\{2,3\}&\{2,4\}&\{3,4\}\\ \hline
\by^{\mathbf b(T)}&1&y_2&y_2y_3&y_1y_2&y_1y_2y_3&y_1y_2^2y_3.
\end{array}
\]
Thus
\[
H=1+y_2+y_2y_3+y_1y_2+y_1y_2y_3+y_1y_2^2y_3,
\qquad S=1-x_1H-x_2,
\]
expression of $S$ follows from Proposition \ref{cor:S}.
Let $\eta=\al_1+2\al_2+\al_3$ and $\mvec=(2,1)$.
The only unordered pairs of monomials in $H$ whose product is
$y_1y_2^2y_3$ are
\[
(1,y_1y_2^2y_3),\quad
(y_2,y_1y_2y_3),\quad
(y_1y_2,y_2y_3).
\]
Here $t=m_1=2$, so need to compute $H^2[y_1y_2^2y_3]$.  Since each pair consists of distinct monomials and therefore contributes twice in $H^2$.
Hence $M_{\omega_2}(\eta;2)=6$, and from \eqref{eq:path-coefficient},
\[
S^q[e^{-(\eta+2\be_1+\be_2)}]
=-6\binom q2\binom{q-2}1=-3q(q-1)(q-2).
\]
Its linear coefficient is $-6$. Since the gcd of the exponent
coordinates is $1$, \eqref{eq:divisor-relation} gives
$\mult(\eta+2\be_1+\be_2)=6$.
\end{example}

\begin{example}
\label{ex:A2}
Take $\mathcal{G}_{\mathrm{re}}=A_2$, $u=2$, $c=1$, $\mathcal{G}_{\mathrm{im}}=L_2$, and $v=1$.
Then $H=1+y_2+y_1y_2$ and $S=1-x_1H-x_2$.
For $\mu=\al_1+2\al_2+2\be_1+2\be_2$,
\[
K_2((1,2);2)=\frac{2!}{1!1!0!}=2,
\qquad
S^q[e^{-\mu}]=2\binom q2\binom{q-2}2
=\frac12q(q-1)(q-2)(q-3),
\]
the above equations follows from \eqref{eq:K} and \eqref{eq:path-coefficient}.
The linear coefficient is $-3$, and the gcd is $1$, so
$\mult(\mu)=3$.
For $\mu'=2\al_1+2\al_2+2\be_1+2\be_2$, on the other hand,
$K_2((2,2);2)=1$ and
\[
C(\mu')=\frac32.
\]
At the divided exponent
$\mu'/2=\al_1+\al_2+\be_1+\be_2$ we have
$K_2((1,1);1)=1$ and $S^q[e^{-\mu'/2}]=q(q-1)$, so
$C(\mu'/2)=1$. The two divisors are $1,2$, with
$\Mob(1)=1$ and $\Mob(2)=-1$. Thus
\[
\mult(\mu')=C(\mu')-\frac12C(\mu'/2)
=\frac32-\frac12=1.
\]

\end{example}

\begin{example}[Type \texorpdfstring{$A_2$}{A2} with the chordal diamond]
\label{ex:A2-diamond}
Use the diamond graph from Example~\ref{ex:diamond}, attached at
$v=1$ to $u=2$ of $A_2$, with $c=1$. Set
$\eta=\al_1+2\al_2$ and $\mvec=(2,1,1,2)$.
Then $K_2((1,2);2)=M(\eta;2)=2$ by Proposition \ref{prop:A-endpoint} and
\[
\pi_{\G_{\mathrm{im}}}^{(2,1,1,2)}(q)
=\binom q2\binom{q-2}1\binom{q-3}1\binom{q-2}2
=\frac14q(q-1)(q-2)^2(q-3)^2,
\]
by equation \eqref{eq:chordal-pi}.
Since $|\mvec|=6$, equation  \eqref{eq:chordal-coefficient} gives
\[
S^q[e^{-(\eta+2\be_1+\be_2+\be_3+2\be_4)}]
=\frac12q(q-1)(q-2)^2(q-3)^2.
\]
Its linear coefficient is $-18$, by evaluating all factors except
the first $q$ at zero. Since gcd is $1$, so the multiplicity is $$C(\eta+2\be_1+\be_2+\be_3+2\be_4)=-S^q[e^{-(\eta+2\be_1+\be_2+\be_3+2\be_4)}][q]=18.$$
\end{example}

\section*{Declaration of generative AI and AI-assisted technologies in the manuscript preparation process}

\noindent
During the preparation of this work, the authors used ChatGPT to refine sentence structure, improve readability, and assist with some computations. The original ideas presented in this paper are those of the authors. Following the use of ChatGPT, the authors reviewed and edited the content as needed and take full responsibility for the content of the published article.


\begin{thebibliography}{100}

\bibitem{AKV}
G. Arunkumar, D. Kus, and R. Venkatesh,
\emph{Root multiplicities for Borcherds algebras and graph coloring},
J. Algebra \textbf{499} (2018), 538--569.

\bibitem{BHV}
D. Biswas, I. Habib, and R. Venkatesh,
\emph{On the characterization of chordal graphs using Horn
hypergeometric series},
Res. Number Theory \textbf{11} (2025), no.~1, Paper No.~20, 14 pp.

\bibitem{Borcherds}
R. E. Borcherds,
\emph{Generalized Kac--Moody algebras},
J. Algebra \textbf{115} (1988), 501--512.


\bibitem{Dirac}
G. A. Dirac,
\emph{On rigid circuit graphs},
Abh. Math. Sem. Univ. Hamburg \textbf{25} (1961), 71--76.

\bibitem{FH}
W. Fulton and J. Harris,
\emph{Representation Theory: A First Course},
Graduate Texts in Mathematics, vol.~129,
Springer-Verlag, New York, 1991.

\bibitem{Golumbic}
M. C. Golumbic,
\emph{Algorithmic Graph Theory and Perfect Graphs},
2nd ed., Annals of Discrete Mathematics, vol.~57,
Elsevier, Amsterdam, 2004.

\bibitem{Kac}
V. G. Kac,
\emph{Infinite-Dimensional Lie Algebras},
3rd ed., Cambridge University Press, Cambridge, 1990.

\bibitem{Kang}
S.-J. Kang,
\emph{Generalized Kac--Moody algebras and the modular function $j$},
Math. Ann. \textbf{298} (1994), 373--384.


\bibitem{Stanley}
R. P. Stanley,
\emph{A symmetric function generalization of the chromatic
polynomial of a graph},
Adv. Math. \textbf{111} (1995), 166--194.

\bibitem{VV}
R. Venkatesh and S. Viswanath,
\emph{Chromatic polynomials of graphs from Kac--Moody algebras},
J. Algebraic Combin. \textbf{41} (2015), no.~4, 1133--1142.

\end{thebibliography}
\end{document}